\RequirePackage{fix-cm}
\documentclass{svjour3}                     % onecolumn (standard format)
\smartqed  % flush right qed marks, e.g. at end of proof
\usepackage{graphicx}
\usepackage{amsmath}
\usepackage{amssymb}
\usepackage{physics}
\usepackage{mathtools}
\usepackage{braket}
\DeclareMathOperator*{\argmax}{arg\,max}
\DeclareMathOperator*{\argmin}{arg\,min}

\begin{document}

\title{Existence and uniqueness of solution of the differential equation describing the TASEP-LK coupled transport process%\thanks{Grants or other notes
}
\subtitle{}


\author{Jingwei Li \and
        Yunxin Zhang %etc.
}


\institute{Jingwei Li \at
              Shanghai Center for Systems Biomedicine, Shanghai Jiao Tong University, Shanghai 200240, China \\
              \email{ljw2017@sjtu.edu.cn}           %  \\
           \and
           Yunxin Zhang \at
              Shanghai Key Laboratory for Contemporary Applied Mathematics, Centre for Computational Systems Biology, School of Mathematical Sciences, Fudan University, Shanghai 200433, China \\
              \email{xyz@fudan.edu.cn} 
}

\date{Received: date / Accepted: date}

\maketitle

\begin{abstract}
We study the existence and uniqueness of solution of a evolutionary partial differential equation originating from the continuum limit of a coupled process of totally asymmetric simple exclusion process (TASEP) and Langmuir kinetics (LK). In the fields of physics and biology, the TASEP-LK coupled process has been extensively studied by Monte Carlo simulations, numerical computations, and detailed experiments. However, no rigorous mathematical analysis so far has been given for the corresponding differential equations, especially the existence and uniqueness of their solutions. In this paper, we prove the existence of the $C^\infty[0,1]$ steady-state solution by the method of upper and lower solution, and the uniqueness in both $W^{1,2}(0,1)$ and $L^\infty(0,1)$ by a generalized maximum principle. We further prove the global existence and uniqueness of the time-dependent solution in $C([0,1]\times [0,+\infty))\cap C^{2,1}([0,1]\times (0,+\infty))$, which, for any continuous initial value, converges to the steady-state solution in $C[0,1]$ (global attractivity). Our results support the numerical calculations and Monte Carlo simulations, and provide theoretical foundations for the TASEP-LK coupled process, especially the most important phase diagram of particle density along the travel track under different model parameters, which is difficult because the boundary layers (at one or both boundaries) and domain wall (separating high and low particle densities) may appear as the length of the travel track tends to infinity. The methods used in this paper may be instructive for studies of the more general cases of the TASEP-LK process, such as the one with multiple travel tracks and/or multiple particle species.
\end{abstract}

\section{Introduction}
\label{intro}

Active transport along filamentous track driven by molecular motors is one of basic mechanisms of intracellular transport, and the case of single isolated motor has been studied extensively \cite{AstumianSymmetry2008,AstumianThermodynamics2010}. In order to describe traffic-like collective movements of many motors simultaneously on the same filamentous track, Aghababaie et al propose a model based on an abstract formulation of Brownian ratchet \cite{AghababaieMenon1999}. Subsequent works however, are generally based on the asymmetric simple exclusion process (ASEP) \cite{SchutzExactly2001}.

In ASEP, two motors cannot occupy the same lattice site (simple exclusion), and motors prefer to move in one direction (asymmetric). ASEP is further specialized to totally asymmetric simple exclusion process (TASEP) by forcing motors to move only in one direction (totally asymmetric). Steady-state solutions of TASEP have been obtained by various methods \cite{Krug1991,DerridaMatrix1993,DerridaRecursion1992,SchutzRecursion1993}.

Most models of molecular motor traffic in practice \cite{LipowskyKlumpp2001,NieuwenhuizenKlumpp2002,KlumppLipowsky2004,KlumppLipowskyTraffic2003,NieuwenhuizenKlumppRandom2004,KlumppNieuwenhuizenSelf2005,LipowskyKlumppLife2005,EvansJuhaszShock2003,JuhaszSantenDynamics2004,Popkov2003} incorporate Langmuir kinetics (LK) that motors can attach and detach filamentous track. Such a TASEP-LK coupled process is deeply discussed in \cite{ParmeggianiPRL2003,Parmeggiani2004,Zhang20101,Zhang2012}. A rich steady-state phase diagram, with high and low density phases, two and three phase coexistence regions, and a boundary independent ``Meissner" phase, is found by considering a continuum limit \cite{ParmeggianiPRL2003,Zhang20101,Zhang2012}. Such profiles of particle density are very different from those of pure TASEP \cite{DerridaMatrix1993,DerridaRecursion1992,SchutzRecursion1993}, which may be considered as the limiting case of TASEP-LK coupled process when attachment and detachment rates of LK tend to zero \cite{Parmeggiani2004}.

The experimental observations of the motor protein Kip3 (in the kinesin-8 family) \cite{LeducGehleMolecular2012} are reproduced by the simulation of Parmeggiani-Franosh-Frey model \cite{ParmeggianiPRL2003}. In \cite{NishinariOkadaIntracellular2005}, the authors introduce a generalized ASEP-LK coupled process, which captures most of the biochemistry of KIF1A motor, and successfully predicts the position of the domain wall in their experiment.

Until now, the steady-state solution of the TASEP-LK coupled process has not been obtained explicitly. The recursion method for pure TASEP \cite{DerridaRecursion1992,SchutzRecursion1993} is too technical to generalize to TASEP-LK coupled process. On the contrary, the matrix product ansatz for pure TASEP \cite{DerridaMatrix1993} is tidy, but the network structure of TASEP-LK coupled process prevents a direct implementation of it \cite{Parmeggiani2004}.

By mean-field approximation \cite{KrugPRL1991}, the TASEP-LK coupled process is transformed to Eq. \eqref{continuumlimitintroduction}, which is a semi-linear initial value parabolic problem with Dirichlet boundary condition \cite{Parmeggiani2004}.

Eq. \eqref{ellipticequationintroduction} is the corresponding time-independent semi-linear elliptic problem of Eq. \eqref{continuumlimitintroduction}, which has been solved numerically in \cite{Parmeggiani2004} and exhibits the same phase diagram as the TASEP-LK coupled process. In this paper, we prove rigorously the results in \cite{Parmeggiani2004} and some further claims.
\begin{enumerate}
\item Eq. \eqref{ellipticequationintroduction} has a unique $C^\infty[0,1]$ solution.
\item The phase diagram of the solution of Eq. \eqref{ellipticequationintroduction} coincides with the numerical one in \cite{ParmeggianiPRL2003,Parmeggiani2004}.
\item Eq. \eqref{continuumlimitintroduction} has a unique $C([0,1]\times [0,+\infty))\cap C^{2,1}([0,1]\times (0,+\infty))$ solution for continuous initial value, which tends to the solution of Eq. \eqref{ellipticequationintroduction} uniformly (global attractivity).
\end{enumerate}

Inspired by the idea used to study a diffusive logistic equation originating from population models in disrupted environments  \cite{Lam2016,SkellamBiometrika1951,CantrellProceedings1989,CantrellSIAM1991,CantrellJoMB1991,LustscherTheoretical2007},
we prove the existence and phase diagram of the solution of Eq. \eqref{ellipticequationintroduction} by the method of upper and lower solution \cite{Du2006}. The uniqueness of the solution of Eq. \eqref{ellipticequationintroduction} is obtained by the comparison principle for divergence form operator \cite{TrudingerArchive1974,Gilbarg2001}. The nonlinear part of Eq. \eqref{continuumlimitintroduction} has divergence form, so {\bf Proposition 7.3.6} in \cite{Lunardi1995} promises the global existence and uniqueness of its solution. The global attractivity is proved by Theorem 3.1 in \cite{Smith1995} since Eq. \eqref{continuumlimitintroduction} is a monotone dynamical system with a unique steady state.

This paper is organized as follows. In Section \ref{model}, we introduce the TASEP-LK coupled process briefly and derive its continuum limit Eq. \eqref{continuumlimitintroduction}. We prove the global existence and uniqueness of the $C([0,1]\times [0,+\infty))\cap C^{2,1}([0,1]\times (0,+\infty))$ solution of Eq. \eqref{continuumlimitintroduction} in Section \ref{partglobalexiuni}, whose global attractivity in $C[0,1]$ is proved by the monotone semiflow theory in Section \ref{partglobalattra}. In Section \ref{uniqueness}, we prove the uniqueness of the $L^\infty(0,1)$ solution of Eq. \eqref{ellipticequationintroduction} by the theory of quasi-linear elliptic equation \cite{Gilbarg2001}, and show that the $L^\infty(0,1)$ solution has $C^{\infty}[0,1]$ regularity. In Sections \ref{phasespesec}, \ref{steadySpecialCases}, \ref{phasegensec}, \ref{steadyGeneralCases}, we use the method of upper and lower solution to prove the existence of a $W^{1,2}(0,1)$ steady-state solution, which has the same phase diagram specified numerically in \cite{Parmeggiani2004}. Finally, conclusions and remarks are presented in Section \ref{conclusion}.

\section{TASEP-LK coupled process}\label{model}
Fig. \ref{TASEP_diagram_figure} gives diagram of the TASEP-LK coupled process.
\begin{figure}
  \centering
  \includegraphics[scale=0.6]{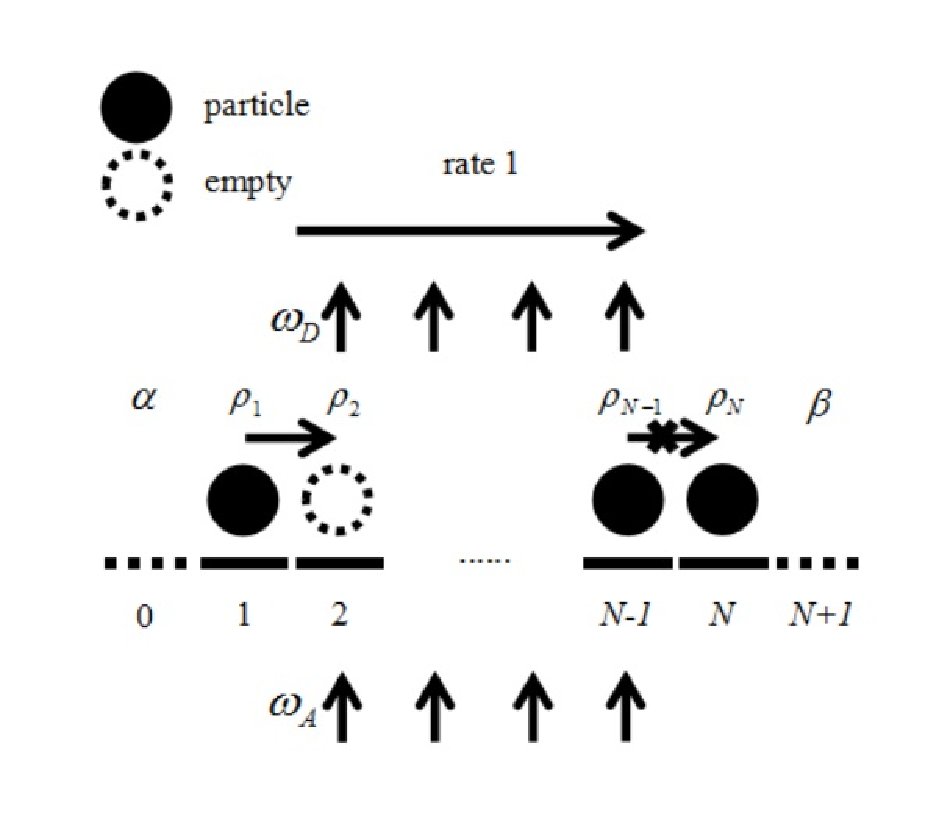}\\
  \caption{A diagram of the TASEP-LK coupled process with $N+2$ sites. Particles move from left to right along a one-dimensional lattice and exclude each other. Sites $0$ and $N+1$ have fixed particle densities $\alpha$ and $\beta$, respectively. Particles at site $i$ hop forward to site $i+1$ if site $i+1$ is vacant (TASEP). Particles can attach and detach the main body of the lattice with rates $\omega_A$ and $\omega_D$, respectively (LK).}
  \label{TASEP_diagram_figure}
\end{figure}
In TASEP, particles of the same species hop unidirectionally with constant rate (usually normalized to be unit) and spatial exclusion (particle at site $i$ can hop to site $i+1$ only if site $i+1$ is empty), along a one-dimensional lattice of $N+2$ sites, labeled from $0$ to $N+1$. In LK, particles attach and detach the main body of the lattice (site $1$ to site $N$) with rates $\omega_A$ and $\omega_D$ respectively \cite{Parmeggiani2004,Zhang2012}. Denote $\epsilon\equiv 1/(N+1)$. Initially, site $i$ is occupied with probability $\sigma(\epsilon i)$. Let $\phi_i(t)$ for $i\in[0,N+1]$ be the probability that site $i$ is occupied at time $t$. At boundaries, $\phi_0(t)\equiv \alpha$ and $\phi_{N+1}(t)\equiv \overline{\beta}\equiv 1-\beta$. By mean-field approximation \cite{KrugPRL1991}, $\phi_i(t)$ satisfies \cite{Parmeggiani2004,Zhang2012}
\begin{eqnarray}\label{TASEPmeanfield2}
\left\{\begin{array}{ll}
\frac{d\phi_i}{dt}=\phi_{i-1}(1-\phi_i)-\phi_i(1-\phi_{i+1})+\omega_A(1-\phi_i)-\omega_D\phi_i, & 1\le i\le N\land t>0, \\
\phi_0=\alpha, \quad \phi_{N+1}=\overline{\beta}, & t>0, \\
\phi_i(0)=\sigma(\epsilon i), & 0\le i\le N+1.
\end{array}\right.
\end{eqnarray}
$\forall i\in[0,N+1]$, $\phi(x,t)\equiv\phi_i(t)$, where $x\equiv \epsilon i$. Assume that $\Omega_A\equiv\omega_A/\epsilon$ and $\Omega_D\equiv\omega_D/\epsilon$ are nonzero constants ($\omega_A$ and $\omega_D$ are of order $\epsilon^\gamma$ with $\gamma=1$) because for $\gamma\ne1$, the TASEP-LK coupled process reduces to either pure TASEP or LK process \cite{Popkov2003,Zhang20131}.

Substitute $\phi(x\pm\epsilon,t)=\phi(x,t)\pm\epsilon\phi_x(x,t)+\frac{1}{2}\epsilon^2\phi_{xx}(x,t)+O(\epsilon^3)$ into Eq. \eqref{TASEPmeanfield2}.
\begin{eqnarray}\label{continuumlimittape}
\left\{\begin{array}{ll}
\phi_t=\epsilon[\frac{\epsilon}{2}\phi_{xx}+(2\phi-1)\phi_x+\Omega_A(1-\phi)-\Omega_D\phi]+O(\epsilon^3), & \frac{1}{N+1}\le x \le \frac{N}{N+1}, \\
 & t>0, \\
\phi(0,t)=\alpha, \quad \phi(1,t)=\overline{\beta}, & t>0, \\
\phi(x,0)=\sigma(x), & 0\le x\le 1.
\end{array}\right.
\end{eqnarray}
For $\epsilon\ll 1$ ($N\gg 1$), neglecting $O(\epsilon^3)$ in Eq. \eqref{continuumlimittape}, we have
\begin{eqnarray}\label{continuumlimitintroduction}
\left\{\begin{array}{ll}
\phi_t=\epsilon\qty[\frac{\epsilon}{2}\phi_{xx}+(2\phi-1)\phi_x+\Omega_A(1-\phi)-\Omega_D\phi], & 0< x <1\land t>0, \\
\phi(0,t)=\alpha,\quad \phi(1,t)=\overline{\beta}, & t>0, \\
\phi(x,0)=\sigma(x), & 0\le x \le1.
\end{array}\right.
\end{eqnarray}
The corresponding time-independent semi-linear elliptic problem of Eq. \eqref{continuumlimitintroduction} is
\begin{eqnarray}\label{ellipticequationintroduction}
\left\{\begin{array}{ll}
L\rho\equiv\frac{\epsilon}{2}\rho_{xx}+(2\rho-1)\rho_x+\Omega_A(1-\rho)-\Omega_D\rho=0, & 0< x <1, \\
\rho(0)=\alpha,\quad \rho(1)=\overline{\beta}. &
\end{array}\right.
\end{eqnarray}

A particle entering the lattice from the left end is always accompanied by a hole leaving the lattice from the left end, and a particle leaving the lattice from the right end is always accompanied by a hole entering the lattice from the right end; a particle hopping right along the lattice is always accompanied by a hole hopping left along the lattice; a particle attaching the lattice is always accompanied by a hole detaching the lattice, and a particle detaching the lattice is always accompanied by a hole attaching the lattice. This is called particle-hole symmetry. Mathematically, $\overline{\phi}(\overline{x},t)\equiv 1-\phi(x,t)$ is the hole density, where $\overline{x}\equiv 1-x$. By Eq. \eqref{continuumlimitintroduction}, $\overline{\phi}$ satisfies
\begin{eqnarray}\label{continuumlimitintroductionsymhole}
\left\{\begin{array}{ll}
\overline{\phi}_t=\epsilon[\frac{\epsilon}{2}\overline{\phi}_{\overline{x}\overline{x}}+(2\overline{\phi}-1)\overline{\phi}_{\overline{x}}+\Omega_D(1-\overline{\phi})-\Omega_A\overline{\phi}], \quad & 0< \overline{x} <1,\ t>0, \\
\overline{\phi}(0,t)=\beta,\quad \overline{\phi}(1,t)=1-\alpha, & t>0, \\
\overline{\phi}(\overline{x},0)=1-\sigma(1-\overline{x}), & 0\le \overline{x} \le1.
\end{array}\right.
\end{eqnarray}
Eq. \eqref{continuumlimitintroductionsymhole} has the same form as Eq. \eqref{continuumlimitintroduction}, but with $\Omega_A$, $\Omega_D$, $\alpha$, $\beta$ replaced by $\Omega_D$, $\Omega_A$, $\beta$, $\alpha$. Particle-hole symmetry allows one to assume $K\equiv\Omega_A/\Omega_D\ge 1$ without loss of generality, since otherwise, one may study $\overline{\phi}$ instead of $\phi$. Particularly, if $K\equiv\Omega_A/\Omega_D=1$, one may assume $\alpha\ge \beta$.

\section{Global existence and uniqueness of the $C([0,1]\times [0,+\infty))\cap C^{2,1}([0,1]\times (0,+\infty))$ solution of Eq. \eqref{continuumlimitintroduction}}
\label{partglobalexiuni}

\begin{theorem}\label{TheSteadySolu}
Eq. \eqref{ellipticequationintroduction} has a $C^\infty[0,1]$ solution $\rho^\epsilon$, which is unique in $L^\infty(0,1)$.
\end{theorem}
\begin{proof}
By Theorems \ref{TheEqOmega} and \ref{TheGeOmega}, Lemma \ref{LemW12Cinftyregularity}, Eq. \eqref{ellipticequationintroduction} has a $C^\infty[0,1]$ solution $\rho^\epsilon$. By Lemmas \ref{C1uniqueness} and \ref{LemLinftyCinftyregularity}, $\rho^\epsilon$ is unique in $L^\infty(0,1)$.
\end{proof}

\begin{theorem}\label{GlobalExsitenceUniquenessTheoofg}
$\forall \kappa\in C_0[0,1]$, the homogeneous Dirichlet boundary problem
\begin{eqnarray}\label{differenceparabolicsectorialform}
\left\{\begin{array}{ll}
\phi_t=\mathcal{A}\phi+\epsilon(\phi^2)_x, & t>0,\ 0<x<1,\\
\phi(0,t)=0,\ \phi(1,t)=0, & t>0, \\
\phi(x,0)=\kappa(x), & 0\le x\le 1,
\end{array}\right.
\end{eqnarray}
with
\begin{eqnarray*}
\mathcal{A}\phi\equiv\frac{\epsilon^2}{2}\phi_{xx}+\epsilon (2\rho^\epsilon-1)\phi_x+\epsilon [2\rho^\epsilon_x-(K+1)\Omega_D]\phi
\end{eqnarray*}
has a unique solution $\eta^{\epsilon,\kappa}\in C([0,1]\times [0,+\infty))\cap C^{2,1}([0,1]\times (0,+\infty))$. $t^{1/2}\eta^{\epsilon,\kappa}_x$ is bounded near $t=0$. The mapping $\Phi_0(t,\kappa)\coloneqq \eta^{\epsilon,\kappa}(\cdot,t):[0,+\infty)\times C_0[0,1]\mapsto C_0[0,1]$ is continuous ($\Phi_0(t,\kappa)$ is semiflow in $C_0[0,1]$). Fixing $\tau>0$, the mapping $\Phi_0(\tau,\kappa):C_0[0,1]\mapsto C_0[0,1]\cap C^1[0,1]$ is continuous.
\end{theorem}
\begin{proof}
Eq. \eqref{differenceparabolicsectorialform} is a second order equation with nonlinearities in divergence form. $\mathcal{A}$ is a time-independent elliptic second order differential operator with continuous coefficients, and $\eta\mapsto \eta^2:\mathbb{R}\mapsto \mathbb{R}$ is twice continuously differentiable. By {\bf Proposition 7.3.6} in \cite{Lunardi1995}, Eq. \eqref{differenceparabolicsectorialform} has a unique solution $\eta^{\epsilon,\kappa}\in C([0,1]\times [0,+\infty))$, $\eta^{\epsilon,\kappa}_x,\eta^{\epsilon,\kappa}_t,\mathcal{A}\eta^{\epsilon,\kappa}\in C([0,1]\times (0,+\infty))$. Therefore,
\begin{eqnarray*}
&&\frac{\epsilon^2}{2}\eta^{\epsilon,\kappa}_{xx}=\mathcal{A}\eta^{\epsilon,\kappa}-\epsilon (2\rho^\epsilon-1)\eta^{\epsilon,\kappa}_x-\epsilon [2\rho^\epsilon_x-(K+1)\Omega_D]\eta^{\epsilon,\kappa}\\
&&\in C([0,1]\times (0,+\infty)).
\end{eqnarray*}
In summary, $\eta^{\epsilon,\kappa}\in C^{2,1}([0,1]\times (0,+\infty))$.

Boundness of $t^{1/2}\eta^{\epsilon,\kappa}_x$ near $t=0$ and continuity of $\Phi_0(t,\kappa):[0,+\infty)\times C_0[0,1]\mapsto C_0[0,1]$ also come from {\bf Proposition 7.3.6} in \cite{Lunardi1995}. Continuity of $\Phi_0(\tau,\kappa):C_0[0,1]\mapsto C_0[0,1]\cap C^1[0,1]$ for fixed $\tau>0$ comes from estimate (7.1.18) of {\bf Theorem 7.1.5} in \cite{Lunardi1995}.
\end{proof}

\begin{corollary}\label{GlobalExsitenceUniquenessTheoofrho}
$\forall \sigma\in X\coloneqq\set{f\in C[0,1] \mid f(0)=\alpha,f(1)=1-\beta}$, Eq. \eqref{continuumlimitintroduction} has a unique solution $\phi^{\epsilon,\sigma}\in C([0,1]\times [0,+\infty))\cap C^{2,1}([0,1]\times (0,+\infty))$. $t^{1/2}\phi^{\epsilon,\sigma}_x$ is bounded near $t=0$. The mapping $\Phi(t,\sigma)\coloneqq\phi^{\epsilon,\sigma}(\cdot,t):[0,+\infty)\times X\mapsto X$ is continuous ($\Phi(t,\sigma)$ is semiflow in $X$). Fixing $\tau>0$, the mapping $\Phi(\tau,\sigma):X\mapsto X\cap C^1[0,1]$ is continuous.
\end{corollary}
\begin{proof}
Note that $\phi^{\epsilon,\sigma}$ solves Eq. \eqref{continuumlimitintroduction} iff $\phi^{\epsilon,\sigma}-\rho^\epsilon=\eta^{\epsilon,\sigma-\rho^\epsilon}$, which solves Eq. \eqref{differenceparabolicsectorialform} with $\kappa=\sigma-\rho^\epsilon$. The proof is then straight forward.
\end{proof}

\section{Global attractivity of Eq. \eqref{continuumlimitintroduction} in $C[0,1]$}
\label{partglobalattra}

\begin{theorem}\label{Globalattractivityofrho}
$\forall \sigma\in X$, $\lim_{t\to\infty}\norm{\Phi(t,\sigma)-\rho^\epsilon}_{C[0,1]}=0$.
\end{theorem}

To prove Theorem \ref{Globalattractivityofrho}, we first prove some lemmas.
\begin{lemma}\label{boundorbit}
$\forall B\subset X$, if $\exists M_u,M_l\in\mathbb{R}$, $\forall \sigma\in B$, $M_l\le\sigma\le M_u$ in $[0,1]$ ($B$ is bounded in $X$), then $\forall \sigma\in B$, $\forall t\ge 0$, $\forall x\in[0,1]$, $\min(M_l,0)\le\Phi(t,\sigma)\le \max(M_u,1)$ (the orbit of $B$ is bounded in $X$).
\end{lemma}
\begin{proof}
$\phi^u\coloneqq\phi^{\epsilon,\sigma}-\max(M_u,1)$ satisfies
\begin{eqnarray}\label{differenceparabolicformrhoum}
\left\{\begin{array}{ll}
\phi^u_t\le\frac{\epsilon^2}{2}\phi^u_{xx}+\epsilon[2\phi^u+2\max(M_u,1)-1]\phi^u_x-\epsilon(K+1)\Omega_D\phi^u, & t>0,\\
 & 0<x<1,\\
\phi^u(0,t)\le0,\ \phi^u(1,t)\le0, & t>0, \\
\phi^u(x,0)\le 0, & 0\le x\le 1.
\end{array}\right.
\end{eqnarray}
$\forall T>0$, prove by contradiction that $\forall (x,t)\in[0,1]\times[0,T]$, $\phi^u=\phi^{\epsilon,\sigma}-\max(M_u,1)\le 0$. Otherwise, $(x_0,t_0)\in\argmax_{(x,t)\in[0,1]\times [0,T]}\phi^u$ satisfies $\phi^u(x_0,t_0)>0$. Since $\phi^u\le 0$ on the parabolic boundary $\set{0}\times[0,T]\cup\set{1}\times[0,T]\cup[0,1]\times\set{0}$, we have $(x_0,t_0)\in (0,1)\times (0,T]$. Therefore, $\phi^u_{xx}(x_0,t_0)\le 0$, $\phi^u_x(x_0,t_0)=0$, $\phi^u_t(x_0,t_0)\ge 0$. By Eq. \eqref{differenceparabolicformrhoum},
\begin{eqnarray*}
&&0\le\phi^u_t(x_0,t_0)\\
&&\le\qty{\frac{\epsilon^2}{2}\phi^u_{xx}+\epsilon[2\phi^u+2\max(M_u,1)-1]\phi^u_x-\epsilon(K+1)\Omega_D\phi^u}(x_0,t_0)<0,
\end{eqnarray*}
conflicts. Similarly, $\forall (x,t)\in[0,1]\times[0,T]$, $\phi^l\coloneqq\phi^{\epsilon,\sigma}-\min(M_l,0)\ge 0$. Since $\sigma\in B$ and $T>0$ are arbitrary, the proof is completed.
\end{proof}

\begin{lemma}\label{semiflowcompactness}
$\forall t_0>0$, the mapping $\Phi(t_0,\sigma):X\mapsto X$ is compact.
\end{lemma}
\begin{proof}
$\forall 0<\tau<\tau'<t_0$, $\forall p>2$, define the following mapping:
\begin{enumerate}
  \item\label{LIDPMFlow} $\Phi_1(\rho)=\Phi(\tau,\rho):X\mapsto \Phi_1(X)$.
  \item\label{LIDPMSolve} $\Phi_2(\rho)=\mathcal{T}_{\tau}\phi^{\epsilon,\rho}:\Phi_1(X)\mapsto W^{1,p}((0,1)\times (\tau',t_0))$. $\mathcal{T}_{\tau}\phi(x,t)\equiv\phi(x,t-\tau)$.
  \item\label{LIDPMCompact} $\Phi_3(\phi)=\phi:W^{1,p}((0,1)\times (\tau',t_0))\mapsto C([0,1]\times [\tau',t_0])$.
  \item\label{LIDPMContinue} $\Phi_4(\phi)=\phi(\cdot,t_0):C([0,1]\times [\tau',t_0])\mapsto C[0,1]$.
\end{enumerate}
$\Phi_1$ is bounded by Lemma \ref{boundorbit}. Now prove $\Phi_2$ is bounded. $\forall B\subset \Phi_1(X)$, $\forall \rho\in B$, $\phi=\mathcal{T}_{\tau}\phi^{\epsilon,\rho}-\rho^\epsilon$ uniquely solves
\begin{eqnarray}\label{differenceparabolicformtaubegin}
\left\{\begin{array}{ll}
\phi_t=\frac{\epsilon^2}{2}\phi_{xx}+\epsilon[2\phi+(2\rho^\epsilon-1)]\phi_x+\epsilon[2\rho^\epsilon_x-(K+1)\Omega_D]\phi, & t>\tau,\\
 & 0<x<1,\\
\phi(0,t)=0,\ \phi(1,t)=0, & t>\tau, \\
\phi(x,\tau)=\rho(x)-\rho^\epsilon(x), & 0\le x\le 1,
\end{array}\right.
\end{eqnarray}
in $C([0,1]\times [\tau,t_0])\cap C^{2,1}([0,1]\times (\tau,t_0])$. Because $\rho\in\Phi_1(X)$, we have $\exists \rho'\in X$, $\Phi(\tau,\rho')=\rho$. Thus, $\phi^{\epsilon,\rho'}-\rho^\epsilon$ solves Eq. \eqref{differenceparabolicformtaubegin} in $C^{2,1}([0,1]\times [\tau,t_0])$. By uniqueness,
\begin{eqnarray}\label{EqFromBeg}
\phi=\phi^{\epsilon,\rho'}-\rho^\epsilon\in C^{2,1}([0,1]\times [\tau,t_0]).
\end{eqnarray}
If $\exists M_1>0$, $\forall \rho\in B$, $\norm{\rho}_{C[0,1]}\le M_1$, then by Lemma \ref{boundorbit}, $\exists M_2>0$, $\forall \rho\in B$,
\begin{eqnarray}\label{EqC2Dbound}
\norm{\phi+\rho^\epsilon}_{C([0,1]\times [\tau,t_0])}\equiv\norm{\mathcal{T}_{\tau}\phi^{\epsilon,\rho}}_{C([0,1]\times [\tau,t_0])}=\norm{\phi^{\epsilon,\rho}}_{C([0,1]\times [0,t_0-\tau])}\le M_2,
\end{eqnarray}
thereby $\exists M_3>0$, $\forall \rho\in B$, $\norm{\phi}_{L^p([0,1]\times [\tau,t_0])}\le M_3$. Thus, by Eq. \eqref{EqC2Dbound}, $\exists M_4>0$, $\forall \rho\in B$,
\begin{eqnarray*}
\norm{\epsilon\qty[2\phi+(2\rho^\epsilon-1)]}_{C([0,1]\times [\tau,t_0])}+\norm{\epsilon[2\rho^\epsilon_x-(K+1)\Omega_D]}_{C([0,1]\times [\tau,t_0])}\le M_4.
\end{eqnarray*}
The coefficient of the highest order term is constant in Eq. \eqref{differenceparabolicformtaubegin}. By {\bf Theorem 7.15} in \cite{Lieberman1996}, $\exists M_5>0$, $\forall \rho\in B$,
\begin{eqnarray}\label{EqgCongxxgt}
\norm{\phi_{xx}}_{L^p((0,1)\times (\tau',t_0))}+\norm{\phi_t}_{L^p((0,1)\times (\tau',t_0))}\le M_5\norm{\phi}_{L^p((0,1)\times (\tau,t_0))}\le M_5M_3.
\end{eqnarray}

Prove the boundness of $\phi_x$ in $L^p((0,1)\times (\tau',t_0))$ by Sobolev interpolation theorem ({\bf Theorem 5.2} in \cite{AdamsFournier2003}). Because $\phi\in C^{2,1}([0,1]\times [\tau,t_0])$, we have $\forall \rho\in B$, $\forall t\in[\tau',t_0]$, $\phi(\cdot,t)\in W^{2,p}(0,1)$. Thus, $\exists M_6>0$, $\forall \rho\in B$, $\forall t\in[\tau',t_0]$,
\begin{eqnarray*}
&&\norm{\phi_x(\cdot,t)}_{L^p(0,1)}^p\le M_6\qty[\norm{\phi_{xx}(\cdot,t)}_{L^p(0,1)}+\norm{\phi(\cdot,t)}_{L^p(0,1)}]^p\\
&&\le 2^{p-1}M_6\qty[\norm{\phi_{xx}(\cdot,t)}_{L^p(0,1)}^p+\norm{\phi(\cdot,t)}_{L^p(0,1)}^p].
\end{eqnarray*}
Thus, by Eq \eqref{EqgCongxxgt}, $\forall\rho\in B$,
\begin{eqnarray*}
&&\norm{\phi_x}_{L^p((0,1)\times (\tau',t_0))}^p=\int_{\tau'}^{t_0}\norm{\phi_x(\cdot,t)}_{L^p(0,1)}^pdt\\
&&\le 2^{p-1}M_6\qty(\int_{\tau'}^{t_0}\norm{\phi_{xx}(\cdot,t)}_{L^p(0,1)}^pdt+\int_{\tau'}^{t_0}\norm{\phi(\cdot,t)}_{L^p(0,1)}^pdt)\\
&&=2^{p-1}M_6\qty[\norm{\phi_{xx}}^p_{L^p((0,1)\times (\tau',t_0))}+\norm{\phi}^p_{L^p((0,1)\times (\tau',t_0))}]\le 2^{p-1}M_3^p(M_5^p+1)M_6.
\end{eqnarray*}
Consequently, $\exists M_7>0$, $\forall \rho\in B$, 
\begin{eqnarray*}
\norm{\mathcal{T}_{\tau}\phi^{\epsilon,\rho}}_{W^{1,p}((0,1)\times (\tau',t_0))}\le\norm{\phi}_{W^{1,p}((0,1)\times (\tau',t_0))}+\norm{\rho^\epsilon}_{W^{1,p}((0,1)\times (\tau',t_0))}\le M_7,
\end{eqnarray*}
thereby $\Phi_2$ is bounded.

Since $p>2$, by Sobolev compact imbedding theorem ({\bf Theorem 6.3} in \cite{AdamsFournier2003}), $\Phi_3$ is compact. The restriction mapping $\Phi_4$ is obviously continuous. By Eq. \eqref{EqFromBeg}, $\Phi_4\circ\Phi_3\circ\Phi_2\circ\Phi_1(\sigma)=\Phi(t_0,\sigma)$, so $\Phi(t_0,\sigma)$ is compact.
\end{proof}

\begin{lemma}\label{semiflowmonotonicity}
$\forall \sigma_1,\sigma_2\in X$, if $\forall x\in[0,1]$, $\sigma_1\ge \sigma_2$, then $\forall T>0$, $\forall x\in[0,1]$, $\Phi(T,\sigma_1)\ge \Phi(T,\sigma_2)$ (monotonicity).
\end{lemma}
\begin{proof}
$\forall T>0$, by Corollary \ref{GlobalExsitenceUniquenessTheoofrho}, $\exists M>0$, $\forall(x,t)\in[0,1]\times (0,T]$, $\abs{t^{1/2}\phi^{\epsilon,\sigma_2}_x}< M$. $\phi\coloneqq\phi^{\epsilon,\sigma_1}-\phi^{\epsilon,\sigma_2}$ satisfies
\begin{eqnarray*}
\left\{\begin{array}{ll}
\phi_t=\frac{\epsilon^2}{2}\phi_{xx}+\epsilon(2\phi^{\epsilon,\sigma_1}-1)\phi_x+\epsilon[2\phi^{\epsilon,\sigma_2}_x-(K+1)\Omega_D]\phi, & t>0,\ 0<x<1,\\
\phi(0,t)=0,\ \phi(1,t)=0, & t>0, \\
\phi(x,0)=\sigma_1(x)-\sigma_2(x)\ge 0, & 0\le x\le 1.
\end{array}\right.
\end{eqnarray*}
Prove by contradiction that $\forall (x,t)\in[0,1]\times[0,T]$, $\phi=\phi^{\epsilon,\sigma_1}-\phi^{\epsilon,\sigma_2}\ge 0$. Otherwise,
\begin{eqnarray*}
(x_0,t_0)\in\argmin_{(x,t)\in [0,1]\times [0,T]}\exp(-4\epsilon Mt^{1/2})\phi
\end{eqnarray*}
satisfies $\phi(x_0,t_0)<0$. Because $\phi\ge 0$ on the parabolic boundary $\set{0}\times[0,T]\cup\set{1}\times[0,T]\cup[0,1]\times\set{0}$, we have $(x_0,t_0)\in (0,1)\times (0,T]$. Thus, $\phi_x(x_0,t_0)=0$, and $\phi_{xx}(x_0,t_0)\ge0$. In summary,
\begin{eqnarray*}
&&0\ge\qty[\exp(-4\epsilon Mt^{1/2})\phi]_t(x_0,t_0)\\
&&=-2\epsilon Mt_0^{-1/2}\exp(-4\epsilon Mt_0^{1/2})\phi(x_0,t_0)+\exp(-4\epsilon Mt_0^{1/2})\phi_t(x_0,t_0)\\
&&=\exp(-4\epsilon Mt_0^{1/2})\bigg\{\frac{\epsilon^2}{2}\phi_{xx}+\epsilon(2\phi^{\epsilon,\sigma_1}-1)\phi_x\\
&&+\epsilon\qty[2t_0^{-1/2}\qty[t_0^{1/2}\phi^{\epsilon,\sigma_2}_x-M]-(K+1)\Omega_D]\phi\bigg\}(x_0,t_0)>0,
\end{eqnarray*}
conflicts. Since $T>0$ is arbitrary, the proof is completed.
\end{proof}

\begin{lemma}\label{semiflowSOP}
$\forall \sigma_1,\sigma_2\in X$, if $\forall x\in[0,1]$, $\sigma_1\ge \sigma_2$, and $\exists x_0\in[0,1]$, $\sigma_1(x_0)> \sigma_2(x_0)$, then $\forall t>0$, $\exists U,V\subset X$, $U$ and $V$ are open sets, $\sigma_1\in U$, $\sigma_2\in V$, $\forall \sigma_u\in U$, $\forall\sigma_v\in V$, $\forall x\in[0,1]$, $\Phi(t,\sigma_u)\ge \Phi(t,\sigma_v)$ (strong order preserving (SOP)).
\end{lemma}
\begin{proof}
Because $\sigma_1(x_0)> \sigma_2(x_0)$, by continuity, $\exists \tau>0$, $\phi^{\epsilon,\sigma_1}(x_0,\tau)> \phi^{\epsilon,\sigma_2}(x_0,\tau)$. Because $\forall x\in[0,1]$, $\sigma_1\ge \sigma_2$, by Lemma \ref{semiflowmonotonicity}, $\forall x\in[0,1]$, $\phi^{\epsilon,\sigma_1}(x,\tau)\ge \phi^{\epsilon,\sigma_2}(x,\tau)$. $\forall T>\tau$, $\exists \lambda>0$, $\lambda\ge \abs{\epsilon\qty[2(\phi^{\epsilon,\sigma_2}_x-(K+1)\Omega_D]}$ in $[0,1]\times [\tau,T]$ (boundness). Then $\phi=\exp(-\lambda t)\qty(\phi^{\epsilon,\sigma_1}-\phi^{\epsilon,\sigma_2})$ satisfies
\begin{eqnarray*}
\left\{\begin{array}{ll}
\phi_t=\frac{\epsilon^2}{2}\phi_{xx}+\epsilon(2\phi^{\epsilon,\sigma_1}-1)\phi_x+\{\epsilon[2\phi^{\epsilon,\sigma_2}_x-(K+1)\Omega_D]-\lambda\}\phi, & t>\tau,\\
 & 0<x<1,\\
\phi(0,t)=0,\ \phi(1,t)=0, & t>\tau, \\
\phi(x,\tau)\ge 0, & 0\le x\le 1.
\end{array}\right.
\end{eqnarray*}
Since in $[0,1]\times [\tau,T]$, $\epsilon(2\phi^{\epsilon,\sigma_1}-1)$ and $\epsilon[2\phi^{\epsilon,\sigma_2}_x-(K+1)\Omega_D]-\lambda$ are bounded and $\epsilon[2\phi^{\epsilon,\sigma_2}_x-(K+1)\Omega_D]-\lambda\le 0$, by strong maximum principle ({\bf Theorem 2.9} in \cite{Lieberman1996}), $\forall (x,t)\in(0,1)\times (\tau,T]$, $\phi>0$ because otherwise, $\exists(x_1,t_1)\in(0,1)\times (\tau,T]$, $\phi(x_1,t_1)=0$, thereby $\forall (x,t)\in[0,1]\times [\tau,t_1]$, $\phi=0$. Therefore,
\begin{eqnarray*}
\qty(\phi^{\epsilon,\sigma_1}-\phi^{\epsilon,\sigma_2})(x_0,\tau)=\exp(\lambda \tau)\phi(x_0,\tau)=0,
\end{eqnarray*}
conflicts. Because $\forall (x,t)\in\set{0,1}\times [\tau,T]$, $\phi=0$, by {\bf Lemma 2.6} in \cite{Lieberman1996}, $\forall t\in(\tau,T]$, $\phi_x(0,t)>0>\phi_x(1,t)$. Since $T>\tau>0$ are arbitrary, we have $\forall t>0$, $\forall x\in(0,1)$, $\phi^{\epsilon,\sigma_1}>\phi^{\epsilon,\sigma_2}$, $\phi^{\epsilon,\sigma_1}_x(0,t)>\phi^{\epsilon,\sigma_2}_x(0,t)$, $\phi^{\epsilon,\sigma_1}_x(1,t)<\phi^{\epsilon,\sigma_2}_x(1,t)$. Thus, $\exists U',V'\subset X\cap C^1[0,1]$, $U'$ and $V'$ are open sets, $\Phi(\sigma_1,t)\in U'$ and $\Phi(\sigma_2,t)$, $\forall \sigma'_u\in U'$, $\forall \sigma'_v\in V'$, $\forall x\in[0,1]$, $\sigma'_u\ge \sigma'_v$. By Corollary \ref{GlobalExsitenceUniquenessTheoofrho}, $\sigma\to \Phi(t,\sigma):X\to X\cap C^1[0,1]$ is continuous. Therefore, $\exists U,V\subset X$, $U$ and $V$ are open sets, $\sigma_1\in U$ and $\sigma_2\in V$, $\Phi(t,U)\subset U'$ and $\Phi(t,V)\subset V'$.
\end{proof}

Finally, we prove Theorem \ref{Globalattractivityofrho}.
\begin{proof}
By Lemma \ref{boundorbit}, semiflow $\Phi(t,\sigma)$ has bounded orbits for bounded initial set. By Lemma \ref{semiflowcompactness}, $\Phi(t,\sigma)$ is compact for fixed $t>0$. $\Phi(t,\sigma)$ is SOP by Lemma \ref{semiflowSOP}. By Theorem \ref{TheSteadySolu}, $\Phi(t,\sigma)$ has unique equilibrium $\rho^\epsilon$ in $C[0,1]$, so $\lim_{t\to\infty}\norm{\Phi(t,\sigma)-\rho^\epsilon}_{C[0,1]}=0$ by Theorem 3.1 in \cite{Smith1995}.
\end{proof}

\section{Uniqueness and $C^\infty[0,1]$ regularity of the $L^{\infty}(0,1)$ solution of Eq. \eqref{ellipticequationintroduction}}\label{uniqueness}

\begin{lemma}\label{C1uniqueness}
$C^1[0,1]$ solution of Eq. \eqref{ellipticequationintroduction}, if exists, is unique.
\end{lemma}
The proof follows Theorem 10.7 in \cite{Gilbarg2001}, which is a generalization of the classical linear maximum principle to the quasi-linear case.
\begin{proof}
Eq. \eqref{ellipticequationintroduction} is $[\mathcal{B}(\rho,\rho_x)]_x+\mathcal{C}(\rho)=0$ with $\mathcal{B}(\rho,\rho_x)\equiv\epsilon\rho_x/2+(\rho)^2-\rho$, and $\mathcal{C}(\rho)\equiv-(K+1)\Omega_D\rho+K\Omega_D$. Suppose both $\rho^{\epsilon,0}$ and $\rho^{\epsilon,1}$ are $C^1[0,1]$ solutions of Eq. \eqref{ellipticequationintroduction}, and define $g\coloneqq\rho^{\epsilon,1}-\rho^{\epsilon,0}$,
\begin{eqnarray*}
g^+(x)\coloneqq\left\{\begin{array}{ll}
g(x), & g(x)\ge0, \\
0, & g(x)<0.
\end{array}\right.
\end{eqnarray*}
Then $\forall \delta>0$, $\varphi\coloneqq\frac{g^+}{g^++\delta}\in W^{1,2}_0(0,1)$ satisfies
\begin{eqnarray*}
&&0=\int_0^1\{[\mathcal{B}(\rho^{\epsilon,1},\rho^{\epsilon,1}_x)-\mathcal{B}(\rho^{\epsilon,0},\rho^{\epsilon,0}_x)]\varphi_x-[\mathcal{C}(\rho^{\epsilon,1})-\mathcal{C}(\rho^{\epsilon,0})]\varphi\}dx\\
&&=\int_0^1\bigg\{\qty[\qty(\rho^{\epsilon,1}+\rho^{\epsilon,0}-1)g^++\frac{\epsilon}{2}g^+_x]\qty(\frac{g^+_x}{g^++\delta}-\frac{g^+_xg^+}{(g^++\delta)^2})\\
&&+(K+1)\Omega_Dg^+\frac{g^+}{g^++\delta}\bigg\}dx\\
&&\ge\int_0^1\qty{\qty(\rho^{\epsilon,1}+\rho^{\epsilon,0}-1)\qty[\log\qty(1+\frac{g^+}{\delta})]_x\frac{g^+}{g^++\delta}\delta+\frac{\epsilon}{2}\qty[\log\qty(1+\frac{g^+}{\delta})]_x^2\delta}dx\\
&&\ge\int_0^1\qty{-\abs{\rho^{\epsilon,1}+\rho^{\epsilon,0}-1}\abs{\qty[\log\qty(1+\frac{g^+}{\delta})]_x}\delta+\frac{\epsilon}{2}\qty[\log\qty(1+\frac{g^+}{\delta})]_x^2\delta}dx.
\end{eqnarray*}
Since $\rho^{\epsilon,0},\rho^{\epsilon,1}\in C^1[0,1]$, $\exists\Lambda>0$, $\abs{\rho^{\epsilon,1}+\rho^{\epsilon,0}-1}\le \Lambda$. Then by H\"{o}lder's inequality,
\begin{eqnarray*}
&&\frac{\epsilon}{2\Lambda}\int_0^1\qty[\log\qty(1+\frac{g^+}{\delta})]_x^2dx\le\int_0^1\qty\abs{\qty[\log\qty(1+\frac{g^+}{\delta})]_x}dx\\
&&\le \sqrt{\int_0^1\qty[\log\qty(1+\frac{g^+}{\delta})]_x^2dx}.
\end{eqnarray*}
Thus,
\begin{eqnarray*}
\int_0^1\qty[\log\qty(1+\frac{g^+}{\delta})]_x^2dx\le\qty(\frac{2\Lambda}{\epsilon})^2.
\end{eqnarray*}
By Poincar\'{e}'s inequality,
\begin{eqnarray*}
\int_0^1\qty[\log\qty(1+\frac{g^+}{\delta})]^2dx\le \qty(\frac{2\Lambda}{\epsilon\pi})^2.
\end{eqnarray*}
$g^+=0$ ($\rho^{\epsilon,1}-\rho^{\epsilon,0}=g\le 0$) because otherwise, $\lim_{\delta\to0^+}\int_0^1\qty[\log\qty(1+\frac{g^+}{\delta})]^2dx=+\infty$ since $g^+$ is continuous in $[0,1]$. By symmetry, $\rho^{\epsilon,0}-\rho^{\epsilon,1}\le 0$. In conclusion, $\rho^{\epsilon,0}=\rho^{\epsilon,1}$.
\end{proof}

\begin{lemma}\label{LemW12Cinftyregularity}
$W^{1,2}(0,1)$ solution of Eq. \eqref{ellipticequationintroduction} has $C^\infty[0,1]$ regularity.
\end{lemma}
\begin{proof}
By mathematical induction, assume $\rho^\epsilon$ is a $W^{n,2}(0,1)$ solution of Eq. \eqref{ellipticequationintroduction}. Since $\frac{\epsilon}{2}\rho^\epsilon_{xx}=-(2\rho^\epsilon-1)\rho^\epsilon_x-\Omega_A(1-\rho^\epsilon)+\Omega_D\rho^\epsilon$ and $\rho^\epsilon\rho^\epsilon_x\in W^{n-1,2}(0,1)$ (Theorem 7.4 in \cite{Gilbarg2001}), $\frac{\epsilon}{2}\rho^\epsilon_{xx}\in W^{n-1,2}(0,1)$, thereby $\rho^\epsilon \in W^{n+1,2}(0,1)$. Inductively, $\rho^\epsilon \in W^{k,2}(0,1)$ $\forall k>0$ if $\rho^\epsilon \in W^{1,2}(0,1)$, thereby $\rho^\epsilon \in C^\infty[0,1]$ (Section 7.7 in \cite{Gilbarg2001}).
\end{proof}

\begin{lemma}\label{LemLinftyCinftyregularity}
$L^{\infty}(0,1)$ solution of Eq. \eqref{ellipticequationintroduction} has $C^\infty[0,1]$ regularity.
\end{lemma}
\begin{proof}
By Theorems \ref{TheEqOmega} and \ref{TheGeOmega}, Eq. \eqref{ellipticequationintroduction} has $W^{1,2}(0,1)$ solution $\rho^{\epsilon,0}$, which is $C^\infty[0,1]$ by Lemma \ref{LemW12Cinftyregularity}. Let $\rho^\epsilon\in L^\infty(0,1)$ solve Eq. \eqref{ellipticequationintroduction}. Then $\rho=\rho^\epsilon-\rho^{\epsilon,0}\in L^{\infty}(0,1)$ satisfies
\begin{eqnarray*}
\left\{\begin{array}{ll}
0=\frac{\epsilon}{2}\rho_{xx}+[\rho^2+\rho(2\rho^{\epsilon,0}-1)]_x-(K+1)\Omega_D\rho, & 0<x<1,\\
\rho(0)=0,\ \rho(1)=0. &
\end{array}\right.
\end{eqnarray*}
$\rho^2+\rho(2\rho^{\epsilon,0}-1)\in L^{\infty}(0,1)\subset L^2(0,1)$, so $[\rho^2+\rho(2\rho^{\epsilon,0}-1)]_x\in W^{-1,2}(0,1)$. Then
\begin{eqnarray*}
\frac{\epsilon}{2}\rho_{xx}=-[\rho^2+\rho(2\rho^{\epsilon,0}-1)]_x+(K+1)\Omega_D\rho\in  W^{-1,2}(0,1).
\end{eqnarray*}
Therefore, $\rho\in  W_0^{1,2}(0,1)$. Thus, $\rho^\epsilon=\rho^{\epsilon,0}+\rho\in W^{1,2}(0,1)$. By Lemma \ref{LemW12Cinftyregularity}, $\rho^\epsilon\in C^\infty[0,1]$.
\end{proof}

\section{The method of upper and lower solution}

\begin{definition}
Define the function set $Y$ such that $\rho\in Y$ iff
\begin{enumerate}
\item $\rho\in C[0,1]$.

\item $\exists k>0$, $\exists 0=x_0<x_1<x_2<\cdots<x_k=1$, $\forall i\in[0,k-1]$, $\rho\in C^2[x_i,x_{i+1}]$.
\end{enumerate}
$\forall\epsilon>0$, define $\zeta^\epsilon\in C^\infty_0(-\infty,+\infty)$ as
\begin{eqnarray*}
\zeta^\epsilon(x)\coloneqq\left\{\begin{array}{ll}
\exp(\frac{1}{x^2-\epsilon^2})/\mathbb{Z}, & \abs{x}\le\epsilon,\\
0, & \abs{x}>\epsilon,
\end{array}\right.
\end{eqnarray*}
where
\begin{eqnarray*}
\mathbb{Z}=\int_{-\epsilon}^\epsilon\exp(\frac{1}{x^2-\epsilon^2})dx.
\end{eqnarray*}
\end{definition}

\begin{lemma}\label{upperlowersufficientcondition}
$\rho\in Y$ is a $W^{1,2}(0,1)$ upper (lower) solution of Eq. \eqref{ellipticequationintroduction} if
\begin{enumerate}
\item $\rho(0)\ge (\le) \alpha$, $\rho(1)\ge (\le) \overline{\beta}$.
\item $\forall i=[0,k-1]$, $\forall x\in[x_i,x_{i+1}]$, $\frac{\epsilon}{2}\rho_{xx}+(2\rho-1)\rho_x+\Omega_A(1-\rho)-\Omega_D\rho\le (\ge) 0$.
\item $\forall i=[1,k-1]$, $\rho_x(x_i^-)\ge (\le) \rho_x(x_i^+)$.
\end{enumerate}
\end{lemma}
\begin{proof}
Use integration by parts (see Definition 4.7 in \cite{Du2006} and Lemma 5.2 in \cite{Lam2016}).
\end{proof}

\begin{theorem}\label{upperlower2W12}
Let $m$ be the Lebesgue measure, $\delta,M>0$, $\hat{\rho}\in L^0(0,1)$. If $\exists\epsilon_0>0$, $\forall \epsilon<\epsilon_0$, $\exists\rho^{\epsilon,l},\rho^{\epsilon,u}\in Y$ satisfy the conditions of lower and upper solutions in Lemma \ref{upperlowersufficientcondition}, $\forall x\in(0,1)$, $\rho^{\epsilon,l}\le \rho^{\epsilon,u}$, $\varlimsup_{\epsilon\to 0}m\qty(\abs{\rho^{\epsilon,u/l}-\hat{\rho}}>\delta)\le M$, then $\forall \epsilon<\epsilon_0$, Eq. \eqref{ellipticequationintroduction} has a $W^{1,2}(0,1)$ solution $\rho^\epsilon$, $\varlimsup_{\epsilon\to 0}m\qty(\abs{\rho^\epsilon-\hat{\rho}}>\delta)\le 2M$.
\end{theorem}
\begin{proof}
$\forall x\in (0,1)$, $\forall\xi \in \mathbb{R}$, $\forall t\in[\rho^{\epsilon,l}(x),\rho^{\epsilon,u}(x)]\subset[m,M]$,
\begin{eqnarray*}
&&\abs{p(t,\xi)}\equiv\abs{-(2t-1)\xi-\Omega_A(1-t)+\Omega_Dt} \\
&&\le \max\qty(\abs{2\max_{x\in[0,1]}\rho^{\epsilon,u}-1},\abs{2\min_{x\in[0,1]}\rho^{\epsilon,l}-1})\abs{\xi}\\
&&+(\Omega_A+\Omega_D)\max\qty(\abs{\max_{x\in[0,1]}\rho^{\epsilon,u}},\abs{\min_{x\in[0,1]}\rho^{\epsilon,l}})+\Omega_A.
\end{eqnarray*}
Let $q=2$. $\mathcal{D}(\xi)\equiv\frac{\epsilon}{2}\xi$ is continuous, thereby satisfies the Caratheodory conditions. $\forall \xi\in\mathbb{R}$, $\abs{\mathcal{D}(\xi)}\le \frac{\epsilon}{2}\abs{\xi}^{q-1}$, $\mathcal{D}(\xi)\xi\ge \frac{\epsilon}{2}\abs{\xi}^q$, $\forall \xi'\neq \xi$, $\qty[\mathcal{D}(\xi)-\mathcal{D}(\xi')](\xi-\xi')=\frac{\epsilon}{2}(\xi-\xi')^2>0$. Because Eq. \eqref{ellipticequationintroduction} has the quasi-linear form
\begin{eqnarray*}
\left\{\begin{array}{ll}
-[\mathcal{D}(\rho_x)]_x+p(\rho,\rho_x)=0, & 0<x<1,\\
\rho(0)=\alpha,\quad \rho(1)=\overline{\beta}, &
\end{array}\right.
\end{eqnarray*}
by Theorem 4.9 in \cite{Du2006} and Lemma \ref{upperlowersufficientcondition}, $\forall \epsilon<\epsilon_0$, Eq. \eqref{ellipticequationintroduction} has a $W^{1,2}(0,1)$ solution $\rho^\epsilon$, $\forall x\in[0,1]$, $\rho^{\epsilon,l}\le\rho^\epsilon\le\rho^{\epsilon,u}$. Thus, $\abs{\rho^\epsilon-\hat{\rho}}>\delta\to \abs{\rho^{\epsilon,u}-\hat{\rho}}>\delta\lor\abs{\rho^{\epsilon,l}-\hat{\rho}}>\delta$. By subadditivity,
\begin{eqnarray*}
&&\varlimsup_{\epsilon\to 0}m\qty(\abs{\rho^\epsilon-\hat{\rho}}>\delta)\le \varlimsup_{\epsilon\to 0}\qty(m\qty(\abs{\rho^{\epsilon,u}-\hat{\rho}}>\delta)+m\qty(\abs{\rho^{\epsilon,l}-\hat{\rho}}>\delta))\\
&&\le\varlimsup_{\epsilon\to 0}m\qty(\abs{\rho^{\epsilon,u}-\hat{\rho}}>\delta)+\varlimsup_{\epsilon\to 0}m\qty(\abs{\rho^{\epsilon,l}-\hat{\rho}}>\delta)\le 2M.
\end{eqnarray*}
\end{proof}

\begin{lemma}\label{LemConInMea}
If $\forall\delta>0$, $\exists\epsilon(\delta)>0$, $\forall\epsilon<\epsilon(\delta)$, Eq. \eqref{ellipticequationintroduction} has a $w^{1,2}(0,1)$ solution $\rho^\epsilon$, $\varlimsup_{\epsilon\to 0}m\qty(\abs{\rho^\epsilon-\hat{\rho}}>\delta)\le\delta$, then $\exists\epsilon_0>0$, $\forall\epsilon<\epsilon_0$, Eq. \eqref{ellipticequationintroduction} has a $w^{1,2}(0,1)$ solution $\rho^\epsilon$, $\forall \delta>0$, $\lim_{\epsilon\to 0}m\qty(\abs{\rho^\epsilon-\hat{\rho}}>\delta)=0$, {\it i.e.} $\rho^\epsilon$ converges to $\hat{\rho}$ in (Lebesgue) measure.
\end{lemma}
\begin{proof}
$\forall n\ge 0$, define $\delta_n\coloneq 2^{-n}$. Let $\epsilon_0\coloneqq\epsilon(\delta_0)$. $\forall n\ge 1$, $\epsilon_n\coloneqq\min(\epsilon_{n-1}/2,\epsilon(\delta_{n-1}))$. By assumption, $\forall\epsilon'<\epsilon(\delta_n)$, Eq. \eqref{ellipticequationintroduction} has a $w^{1,2}(0,1)$ solution $\rho^{\epsilon',n}$,
\begin{eqnarray*}
\varlimsup_{\epsilon'\to 0}m\qty(\abs{\rho^{\epsilon',n}-\hat{\rho}}>\delta_n)\le\delta_n.
\end{eqnarray*}
$\forall\epsilon_{n+1}\le\epsilon<\epsilon_n\le\epsilon(\delta_n)$, let $\rho^\epsilon=\rho^{\epsilon,n}$. Then $\forall n\ge 0$, $\forall\delta\ge\delta_n$,
\begin{eqnarray*}
\varlimsup_{\epsilon\to 0}m\qty(\abs{\rho^\epsilon-\hat{\rho}}>\delta)\le\varlimsup_{\epsilon\to 0}m\qty(\abs{\rho^\epsilon-\hat{\rho}}>\delta_n)\le\delta_n.
\end{eqnarray*}
The result follows as $n\to+\infty$.
\end{proof}

\section{Phases of Eq. \eqref{ellipticequationintroduction} as $\epsilon\to 0$ for $\Omega_A=\Omega_D=\Omega$}\label{phasespesec}

By simulations, previous studies \cite{Parmeggiani2004,Zhang20101,Zhang2012} find that as $\epsilon\to 0$, the numerical solution of Eq. \eqref{ellipticequationintroduction} tends to certain phase $\hat{\rho}\in L^0(0,1)$ depending essentially on $\alpha$ and $\beta$ (boundary-induced phase transition). See \cite{Nishinari2005,Leduc2012} for experimental observations. In this part, we summary the phase $\hat{\rho}$ for $\Omega_A=\Omega_D=\Omega$. Now Eq. \eqref{ellipticequationintroduction} becomes
\begin{eqnarray*}
L\rho^\epsilon\equiv\frac{\epsilon}{2}\rho_{xx}+(2\rho-1)(\rho_x-\Omega)=0.
\end{eqnarray*}
By particle-hole symmetry in Section \ref{model}, assume $\alpha\ge \beta$. There are 5 phases depending on $\alpha$, $\beta$, $\Omega$.
\begin{enumerate}
\item If $\alpha>\beta$, $\beta+\Omega>\alpha$, $\alpha+\beta+\Omega<1$, then (Fig. \ref{special_phase}(a))
\begin{eqnarray*}
\hat{\rho}=\left\{\begin{array}{ll}
\Omega x+\alpha, & 0\le x\le \frac{\Omega+\beta-\alpha}{2\Omega}, \\
\overline{\beta}-\Omega+\Omega x, & \frac{\Omega+\beta-\alpha}{2\Omega}< x\le 1.
\end{array}\right.
\end{eqnarray*}

\item If $\alpha>\beta$, $\alpha<0.5$, $\alpha+\beta+\Omega>1$, then (Fig. \ref{special_phase}(b))
\begin{eqnarray*}
\hat{\rho}=\left\{\begin{array}{ll}
\Omega x+\alpha, & 0\le x< \frac{1-2\alpha}{2\Omega}, \\
0.5, & \frac{1-2\alpha}{2\Omega}\le x\le \frac{2\Omega+2\beta-1}{2\Omega}, \\
\overline{\beta}-\Omega+\Omega x, & \frac{2\Omega+2\beta-1}{2\Omega}< x\le 1.
\end{array}\right.
\end{eqnarray*}

\item If $\alpha>0.5>\beta>0.5-\Omega$, then (Fig. \ref{special_phase}(c))
\begin{eqnarray*}
\hat{\rho}=\left\{\begin{array}{ll}
\alpha, & x=0, \\
0.5, & 0<x\le \frac{2\Omega+2\beta-1}{2\Omega}, \\
\overline{\beta}-\Omega+\Omega x, & \frac{2\Omega+2\beta-1}{2\Omega}< x\le 1.
\end{array}\right.
\end{eqnarray*}

\item If $\alpha>\beta+\Omega$, $\beta<0.5-\Omega$, then (Fig. \ref{special_phase}(d,e))
\begin{eqnarray*}
\hat{\rho}=\left\{\begin{array}{ll}
\alpha, & x=0, \\
\overline{\beta}-\Omega+\Omega x, & 0< x\le 1.
\end{array}\right.
\end{eqnarray*}

\item If $\alpha>0.5$, $\beta>0.5$, then (Fig. \ref{special_phase}(f))
\begin{eqnarray*}
\hat{\rho}=\left\{\begin{array}{ll}
\alpha, & x=0, \\
0.5, & 0< x< 1, \\
\overline{\beta}, & x=1.
\end{array}\right.
\end{eqnarray*}
\end{enumerate}

\begin{figure}
  \centering
\includegraphics[scale=0.35]{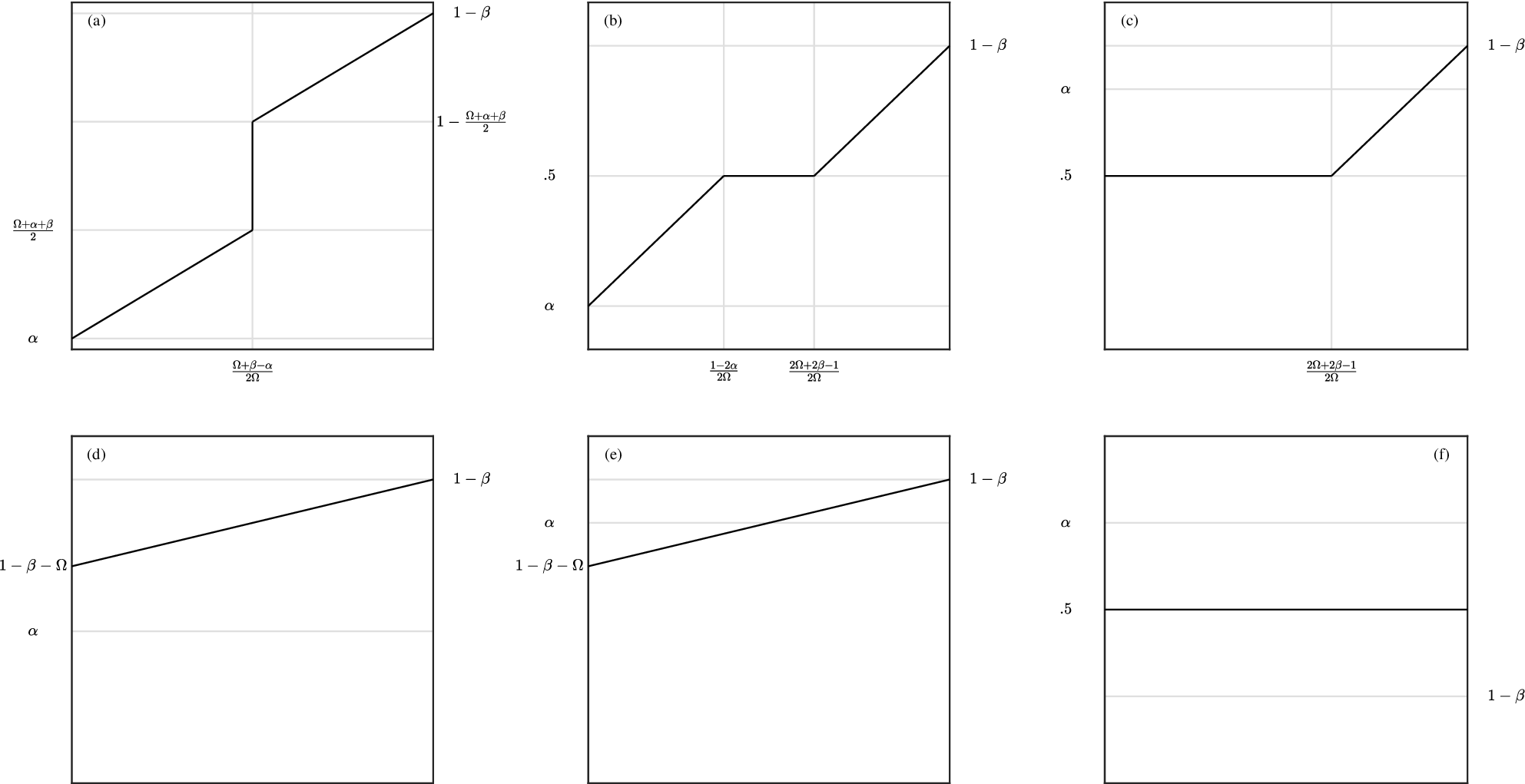}\\
\caption{Diagrams of $\hat{\rho}$ for $\Omega_D=\Omega_A=\Omega$. (a) $\Omega=5/8$, $\alpha=\beta=1/32$. (b) $\Omega=1$, $\alpha=\beta=1/8$. (c) $\Omega=1$, $\alpha=3/4$, $\beta=1/8$. (d) $\Omega=1/4$, $\alpha=7/16$, $\beta=1/8$. (e) $\Omega=1/4$, $\alpha=3/4$, $\beta=1/8$. (f) $\Omega=1/4$, $\alpha=3/4$, $\beta=3/4$.}\label{special_phase}
\end{figure}

Some phases may disappear for specific $\Omega$ values. In Fig. \ref{special_space}, we show four typical incomplete phase diagrams.

\begin{figure}
  \centering
\includegraphics[scale=0.35]{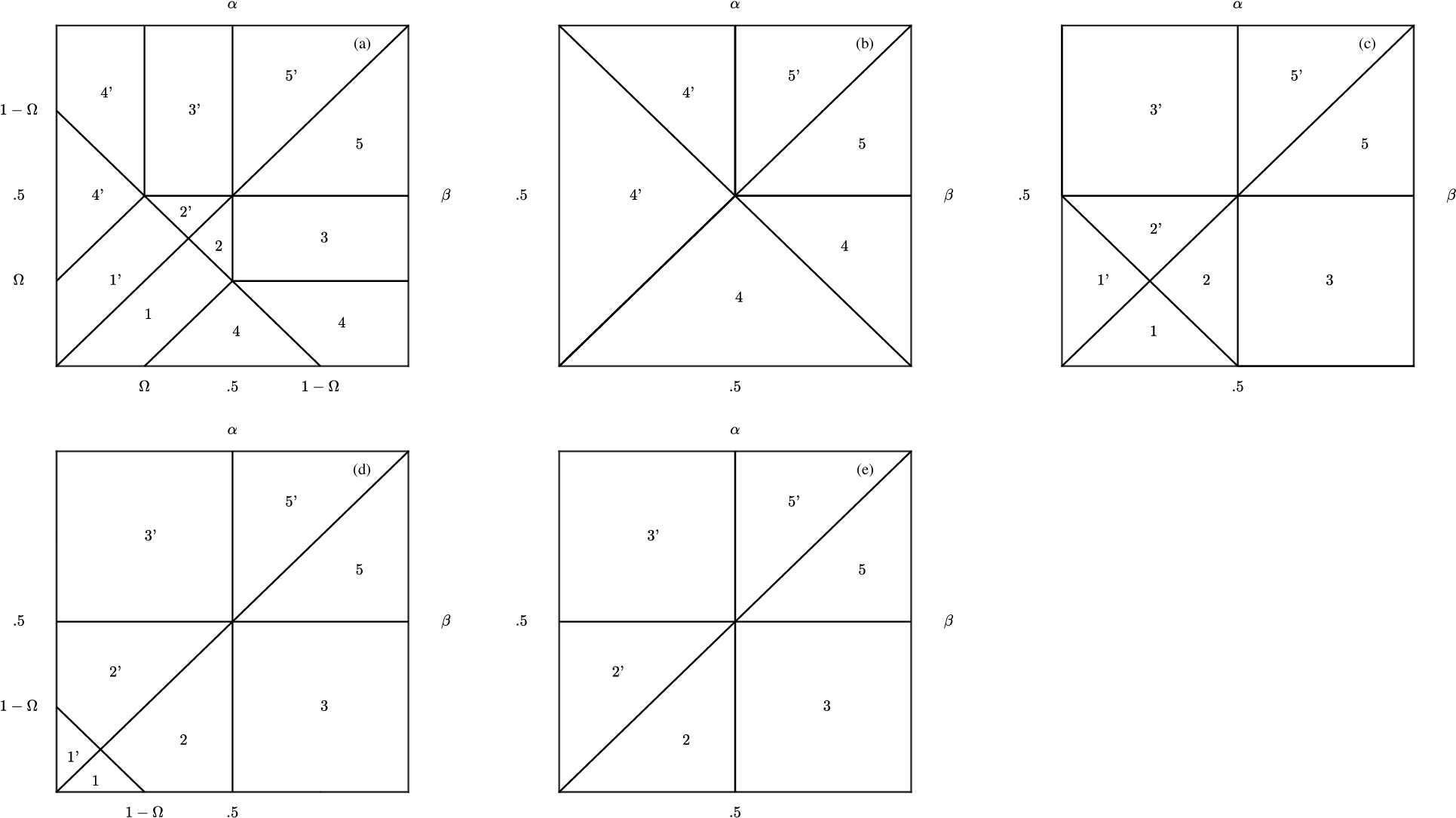}\\
\caption{Phase diagrams of $\hat{\rho}$ for $\Omega_D=\Omega_A=\Omega$. (a) $\Omega=0.25$. (b) $\Omega=0$. (c) $\Omega=0.5$. (d) $\Omega=0.75$. (e) $\Omega\ge 1$.}\label{special_space}
\end{figure}

\begin{theorem}\label{TheEqOmega}
If $\Omega_A=\Omega_D=\Omega$, then $\exists\epsilon_0>0$, $\forall\epsilon<\epsilon_0$, Eq. \eqref{ellipticequationintroduction} has a $w^{1,2}(0,1)$ solution $\rho^\epsilon$, which converges to $\hat{\rho}$ in (Lebesgue) measure.
\end{theorem}

\section{Proof of Theorem \ref{TheEqOmega}}\label{steadySpecialCases}

\subsection{Phase 1: $\alpha>\beta$, $\beta+\Omega>\alpha$, $\alpha+\beta+\Omega<1$}\label{spephase1seclab}

$\forall\delta>0$, define $x_d^u\coloneqq\frac{\Omega+\beta-\alpha}{2\Omega}-\frac{\delta}{2}$. Because
\begin{eqnarray*}
1-\beta=1-\frac{\Omega+\beta+\alpha}{2}+\Omega\qty(1-\frac{\Omega+\beta-\alpha}{2\Omega})<1-\frac{\Omega+\beta+\alpha}{2}+\Omega\qty(1-x_d^u),
\end{eqnarray*}
we have, $\exists 0<\Omega^u<\Omega$,
\begin{eqnarray*}
1-\frac{\Omega+\beta+\alpha}{2}+\Omega^u\qty(1-x_d^u)>1-\beta.
\end{eqnarray*}
Define $\rho^{\epsilon,u}\in Y$ as
\begin{eqnarray*}
\rho^{\epsilon,u}\coloneqq\left\{\begin{array}{ll}
w^\epsilon+\Omega\qty(x-x_d^u), & x\le x_d^u,\\
w^\epsilon+\Omega^u(x-x_d^u), & x>x_d^u,
\end{array}\right.
\end{eqnarray*}
where $x_d^u\coloneqq\frac{\Omega+\beta-\alpha}{2\Omega}-\frac{\delta}{2}$, $w^\epsilon$ solves
\begin{eqnarray*}
\left\{\begin{array}{l}
\frac{\epsilon}{2}w^\epsilon_x=-\qty(w^\epsilon-\frac{\Omega+\beta+\alpha}{2})\qty(w^\epsilon-\qty(1-\frac{\Omega+\beta+\alpha}{2})),\\
w^\epsilon\qty(x_d^u-\frac{\delta}{2})=0.5.
\end{array}\right.
\end{eqnarray*}
$\forall x\le x_d^u$,
\begin{eqnarray}\label{EqSpePhase1L1}
L\rho^{\epsilon,u}=2\Omega\qty(x-x_d^u)w^\epsilon_x\le 0.
\end{eqnarray}
Also,
\begin{eqnarray*}
&&\lim_{\epsilon\to 0}\max_{x\in\qty[x_d^u,1]}(2\rho^{\epsilon,u}-1)(\Omega^u-\Omega)\le \lim_{\epsilon\to 0}\qty[2\rho^{\epsilon,u}\qty(x_d^u)-1](\Omega^u-\Omega)\\
&&=\lim_{\epsilon\to 0}\qty[2w^\epsilon\qty(x_d^u)-1](\Omega^u-\Omega)=(1-\Omega-\alpha-\beta)(\Omega^u-\Omega)<0.
\end{eqnarray*}
Thus,
\begin{eqnarray}\label{EqSpePhase1L2}
&&\lim_{\epsilon\to 0}\max_{x\in\qty[x_d^u,1]}L\rho^{u,\epsilon}\le\lim_{\epsilon\to 0}\max_{x\in\qty[x_d^u,1]}2\Omega^u\qty(x-x_d^u)w^\epsilon_x\nonumber\\
&&+\lim_{\epsilon\to 0}\max_{x\in\qty[x_d^u,1]}(2\rho^{\epsilon,u}-1)(\Omega^u-\Omega)<0.
\end{eqnarray}
Also,
\begin{eqnarray}\label{EqSpePhase1CuspBound}
&&\rho^{\epsilon,u}_x\qty(\qty(x_d^u)^-)=w^\epsilon_x(x_d^u)+\Omega>w^\epsilon_x(x_d^u)+\Omega^u=\rho^{\epsilon,u}_x\qty(\qty(x_d^u)^+),\nonumber\\
&&\rho^{\epsilon,u}(0)=w^\epsilon(0)-\frac{\Omega+\beta+\alpha}{2}+\alpha+\frac{\Omega\delta}{2}>\alpha,\nonumber\\
&&\lim_{\epsilon\to 0}\rho^{\epsilon,u}(1)=\lim_{\epsilon\to 0}w^\epsilon(1)+\Omega^u\qty(1-x_d^u)\nonumber\\
&&=1-\frac{\Omega+\beta+\alpha}{2}+\Omega^u\qty(1-x_d^u)>1-\beta.
\end{eqnarray}
By Eqs. \eqref{EqSpePhase1L1}, \eqref{EqSpePhase1L2}, \eqref{EqSpePhase1CuspBound}, $\exists\epsilon^u_0>0$, $\forall\epsilon<\epsilon^u_0$, $\rho^{\epsilon,u}$ satisfies the conditions of upper solution in Lemma \ref{upperlowersufficientcondition}.

Define
\begin{eqnarray*}
&&w\coloneqq\left\{\begin{array}{ll}
\frac{\Omega+\beta+\alpha}{2}, & x\le x_d^u-\frac{\delta}{2},\\
1-\frac{\Omega+\beta+\alpha}{2}, & x>x_d^u-\frac{\delta}{2},
\end{array}\right.\\
&&\rho^u\coloneqq\left\{\begin{array}{ll}
w+\Omega(x-x_d^u), & x\le x_d^u,\\
w+\Omega^u(x-x_d^u), & x> x_d^u.
\end{array}\right.
\end{eqnarray*}
Then
\begin{eqnarray*}
&&\varlimsup_{\epsilon\to 0}m\qty(\abs{\rho^{\epsilon,u}-\hat{\rho}}>\qty(1+\frac{\Omega}{2})\delta)\\
&&\le\varlimsup_{\epsilon\to 0}\qty(m\qty(\abs{\rho^{\epsilon,u}-\rho^u}>\delta)+m\qty(\abs{\rho^u-\hat{\rho}}>\frac{\Omega\delta}{2}))\\
&&=\varlimsup_{\epsilon\to 0}m\qty(\abs{w^\epsilon-w}>\delta)+m\qty(\abs{\rho^u-\hat{\rho}}>\frac{\Omega\delta}{2})\le\delta.
\end{eqnarray*}

$\forall\delta>0$, define $\rho^{\epsilon,l}$ by symmetry. Because $\rho^{\epsilon,l}<\rho^{\epsilon,u}$, by Theorem \ref{upperlower2W12}, $\forall \epsilon<\min(\epsilon_0^u,\epsilon_0^l)$, Eq. \eqref{ellipticequationintroduction} has a $W^{1,2}(0,1)$ solution $\rho^\epsilon$,
\begin{eqnarray*}
\varlimsup_{\epsilon\to 0}m\qty(\abs{\rho^\epsilon-\hat{\rho}}>\qty(1+\frac{\Omega}{2})\delta)\le 2\delta.
\end{eqnarray*}
Theorem \ref{TheEqOmega} follows Lemma \ref{LemConInMea}.

\subsection{Phase 2: $\alpha>\beta$, $\alpha<0.5$, $\alpha+\beta+\Omega>1$}\label{spephase2seclab}
$\forall\delta>0$, because
\begin{eqnarray*}
1-\beta<1-\beta+\delta=0.5+\delta+\Omega\frac{0.5-\beta}{\Omega},
\end{eqnarray*}
we have $\exists \Omega^u<\Omega$,
\begin{eqnarray*}
0.5+\delta+\Omega^u\frac{0.5-\beta}{\Omega}>1-\beta.
\end{eqnarray*}
Define $\rho^{\epsilon,u}\in Y$ as
\begin{eqnarray*}
\rho^{\epsilon,u}\coloneqq\left\{\begin{array}{ll}
0.5+\delta+\Omega x-0.5+\alpha, & 0\le x\le x_0, \\
f*\zeta^\epsilon, & x_0< x\le 1,
\end{array}\right.
\end{eqnarray*}
where $x_0\coloneqq\frac{0.5-\alpha}{\Omega}$,
\begin{eqnarray*}
f(x)\coloneqq\left\{\begin{array}{ll}
0.5+\delta, & x\le 1-\frac{0.5-\beta}{\Omega},\\
0.5+\delta+\Omega^u \qty(x-1+\frac{0.5-\beta}{\Omega}), & x>1-\frac{0.5-\beta}{\Omega}.
\end{array}\right.
\end{eqnarray*}
$\forall 0\le x\le x_0$,
\begin{eqnarray}\label{EqSpePhase2L1}
L\rho^{\epsilon,u}=0
\end{eqnarray}
Because $f(x)$ is piece-wise linear, we have $\forall x\in\mathbb{R}$,
\begin{eqnarray*}
\qty(f*\zeta^\epsilon)_{xx}(x)=\qty(\int_{-\infty}^{+\infty}f(y+x)\zeta(y)dy)_{xx}=\int_{-\infty}^{+\infty}\qty(f(y+x))_{xx}\zeta(y)dy=0.
\end{eqnarray*}
Because $\forall x\in\mathbb{R}$, $\forall h>0$, $f(x)\ge0.5+\delta$, $f(x+h)-f(x)\le\Omega^uh$, we have
\begin{eqnarray*}
&&f*\zeta^\epsilon(x)=\int_{-\infty}^{+\infty}f(y+x)\zeta(y)dy\le0.5+\delta,\\
&&f*\zeta^\epsilon(x+h)-f*\zeta^\epsilon(x)=\int_{-\infty}^{+\infty}\qty[f(y+x+h)-f(y+x)]\zeta(y)dy\le \Omega^uh.
\end{eqnarray*}
Thus, $(f*\zeta^\epsilon)_x\le\Omega^u$. Therefore, $\forall x_0<x\le 1$.
\begin{eqnarray}\label{EqSpePhase2L2}
L\rho^{\epsilon,u}=\frac{\epsilon}{2} \qty(f*\zeta^\epsilon)_{xx}+(2f*\zeta^\epsilon-1)\qty(\qty(f*\zeta^\epsilon)_x-\Omega)\le 2\delta\qty(\Omega^u-\Omega)<0.
\end{eqnarray}
Also,
\begin{eqnarray}\label{EqSpePhase2CuspBound}
&&\rho^{\epsilon,u}_x(x_0^-)=\Omega>\Omega^u\ge(f*\zeta^\epsilon)_x(x_0)=\rho^{\epsilon,u}_x(x_0^+),\nonumber\\
&&\rho^{\epsilon,u}(0)=\delta+\alpha>\alpha,\nonumber\\
&&\lim_{\epsilon\to0}\rho^{\epsilon,u}(1)=\lim_{\epsilon\to0}f*\zeta^\epsilon(1)=f(1)=0.5+\delta+\Omega^u\frac{0.5-\beta}{\Omega}<1-\beta.
\end{eqnarray}
By Eqs. \eqref{EqSpePhase2L1}, \eqref{EqSpePhase2L2}, \eqref{EqSpePhase2CuspBound}, $\exists\epsilon^u_0>0$, $\forall\epsilon<\epsilon^u_0$, $\rho^{\epsilon,u}$ satisfies the conditions of upper solution in Lemma \ref{upperlowersufficientcondition}.

Define
\begin{eqnarray*}
\rho^u\coloneqq\left\{\begin{array}{ll}
0.5+\delta+\Omega x-0.5+\alpha, & 0\le x\le x_0, \\
f, & x_0< x\le 1.
\end{array}\right.
\end{eqnarray*}
Then
\begin{eqnarray*}
&&\varlimsup_{\epsilon\to 0}m\qty(\abs{\rho^{\epsilon,u}-\hat{\rho}}>2\delta)\le\varlimsup_{\epsilon\to 0}\qty(m\qty(\abs{\rho^{\epsilon,u}-\rho^u}>\delta)+m\qty(\abs{\rho^u-\hat{\rho}}>\delta))\\
&&=\varlimsup_{\epsilon\to 0}m\qty(\abs{\rho^{\epsilon,u}-\rho^u}>\delta)+m\qty(\abs{\rho^u-\hat{\rho}}>\delta)=0.
\end{eqnarray*}

$\forall\delta>0$, define $\rho^{\epsilon,l}$ by symmetry. Because $\rho^{\epsilon,l}<\rho^{\epsilon,u}$, by Theorem \ref{upperlower2W12}, $\forall \epsilon<\min(\epsilon_0^u,\epsilon_0^l)$, Eq. \eqref{ellipticequationintroduction} has a $W^{1,2}(0,1)$ solution $\rho^\epsilon$,
\begin{eqnarray*}
\varlimsup_{\epsilon\to 0}m\qty(\abs{\rho^\epsilon-\hat{\rho}}> 2\delta)=0.
\end{eqnarray*}
Theorem \ref{TheEqOmega} follows Lemma \ref{LemConInMea}.

\subsection{Phase 3: $\alpha>0.5$, $0.5-\Omega<\beta<0.5$}\label{spephase3seclab}
$\forall\delta>0$, because
\begin{eqnarray*}
1-\beta<1-\beta+\delta=0.5+\delta+\Omega\frac{0.5-\beta}{\Omega},
\end{eqnarray*}
we have $\exists 0<\Omega^u<\Omega$,
\begin{eqnarray*}
0.5+\delta+\Omega^u\frac{0.5-\beta}{\Omega}>1-\beta.
\end{eqnarray*}
Define $\rho^{\epsilon,u}\in Y$ as $\rho^{\epsilon,u}\coloneqq f^\epsilon*\zeta^\epsilon$, where $x_0\coloneqq 1-\frac{0.5-\beta}{\Omega}$,
\begin{eqnarray*}
f^\epsilon\coloneqq\left\{\begin{array}{ll}
w^\epsilon, & x\le x_0, \\
w^\epsilon(x_0)+\Omega^u (x-x_0), &  x>x_0,
\end{array}\right.
\end{eqnarray*}
$w^\epsilon$ solves
\begin{eqnarray*}
\left\{\begin{array}{l}
\frac{\epsilon}{2}w^\epsilon_x=-\qty(w^\epsilon-0.5-\delta)\qty(w^\epsilon-0.5+\delta),\\
w^\epsilon(0)=\alpha.
\end{array}\right.
\end{eqnarray*}
$\forall 0\le x\le x_0-\delta$, $\forall\epsilon<\delta$,
\begin{eqnarray*}
&&L\rho^{\epsilon,u}=\int_{-\infty}^{+\infty}Lf(y+x)\zeta^\epsilon(y)dy=\int_{-\infty}^{+\infty}Lw^\epsilon(y+x)\zeta^\epsilon(y)dy\\
&&=-\Omega\int_{-\infty}^{+\infty}\qty(2w^\epsilon(y+x)-1)\zeta^\epsilon(y)dy<-2\Omega\delta\int_{-\infty}^{+\infty}\zeta^\epsilon(y)dy=-2\Omega\delta<0.
\end{eqnarray*}
Because
\begin{eqnarray*}
&&\lim_{\epsilon\to0}\max_{x\ge x_0-2\delta}\frac{\epsilon}{2}f^\epsilon_{xx}=\lim_{\epsilon\to0}\max_{x\in\qty[x_0-2\delta,x_0]}\frac{\epsilon}{2}w^\epsilon_{xx}=-\lim_{\epsilon\to0}\max_{x\in\qty[x_0-2\delta,x_0]}(2w^\epsilon-1)w^\epsilon_x=0,\\
&&\max_{x\ge x_0-2\delta}f^\epsilon_x\le\Omega^u,\\
&&\min_{x\ge x_0-2\delta}f^\epsilon\ge 0.5+\delta,
\end{eqnarray*}
we have $\forall x_0-\delta<x\le 1$,
\begin{eqnarray}\label{EqSpePhase3L2}
&&\varlimsup_{\epsilon\to 0}L\rho^{\epsilon,u}=\varlimsup_{\epsilon\to 0}\int_{-\infty}^{+\infty}Lf^\epsilon(y+x)\zeta^\epsilon(y)dy\nonumber\\
&&=\varlimsup_{\epsilon\to 0}\int_{-\infty}^{+\infty}\qty(\frac{\epsilon}{2}f^\epsilon_{xx}+\qty(2f^\epsilon-1)\qty(f^\epsilon_x-\Omega))(y+x)\zeta^\epsilon(y)dy\nonumber\\
&&\le2\delta\qty(\Omega^u-\Omega)<0.
\end{eqnarray}
Also,
\begin{eqnarray}\label{EqSpePhase3Bound}
&&\lim_{\epsilon\to0}\rho^{\epsilon,u}(0)=\lim_{\epsilon\to0}f^\epsilon(0)=\lim_{\epsilon\to0}w^\epsilon(0)=\alpha,\nonumber\\
&&\lim_{\epsilon\to0}\rho^{\epsilon,u}(1)=\lim_{\epsilon\to0}f^\epsilon(1)=\lim_{\epsilon\to0}w^\epsilon(x_0)+\Omega^u\frac{0.5-\beta}{\Omega}\nonumber\\
&&=0.5+\delta+\Omega^u\frac{0.5-\beta}{\Omega}>1-\beta.
\end{eqnarray}
By Eqs. \eqref{EqSpePhase1L1}, \eqref{EqSpePhase3L2}, \eqref{EqSpePhase3Bound}, $\exists\epsilon_0>0$, $\forall\epsilon<\epsilon_0$, $\rho^{\epsilon,u}$ satisfies the conditions of upper solution in Lemma \ref{upperlowersufficientcondition}.

Define
\begin{eqnarray*}
\rho^u\coloneqq\left\{\begin{array}{ll}
\alpha, & x=0,\\
0.5+\delta, & 0<x\le x_0, \\
0.5+\delta+\Omega^u (x-x_0), &  x>x_0.
\end{array}\right.
\end{eqnarray*}
Then
\begin{eqnarray*}
\varlimsup_{\epsilon\to0}m\qty(\abs{\rho^{\epsilon,u}-\hat{\rho}}>2\delta)\le\varlimsup_{\epsilon\to0}\qty(m\qty(\abs{\rho^u-\hat{\rho}}>\delta)+m\qty(\abs{\rho^{\epsilon,u}-\rho^u}>\delta))=0.
\end{eqnarray*}

Define $\rho^l\in Y$ as
\begin{eqnarray*}
\rho^l\coloneqq\left\{\begin{array}{ll}
\alpha, & x=0, \\
0.5, & 0\le x\le x_0, \\
0.5+\Omega(x-x_0), & x_0<x\le 1.
\end{array}\right.
\end{eqnarray*}
Then
\begin{eqnarray*}
&&L\rho^l=0,\\
&&\rho^l_x(x_0^-)=0<\Omega=\rho^l_x(x_0^+),\\
&&\rho^l(0)=0.5<\alpha,\\
&&\rho^l(1)=1-\beta.
\end{eqnarray*}
Thus, $\forall\epsilon>0$, $\rho^l$ satisfies the conditions of lower solutions in Lemma \ref{upperlowersufficientcondition}. Because $m(\abs{\rho^l-\hat{\rho}}>0)=0$, $\rho^l<\rho^{\epsilon,u}$, by Theorem \ref{upperlower2W12}, we have $\forall\epsilon<\epsilon_0$, Eq. \eqref{ellipticequationintroduction} has a $W^{1,2}(0,1)$ solution $\rho^\epsilon$,
\begin{eqnarray*}
\varlimsup_{\epsilon\to 0}m\qty(\abs{\rho^\epsilon-\hat{\rho}}> 2\delta)=0.
\end{eqnarray*}
Theorem \ref{TheEqOmega} follows Lemma \ref{LemConInMea}.

\subsection{Phase 4: $\alpha>\beta+\Omega$, $0.5>\beta+\Omega$}\label{spephase4seclab}
Define $\rho^a,\rho^{\epsilon,b}\in Y$ as $\rho^a\coloneqq 1-\beta-\Omega+\Omega x$, $\rho^{\epsilon,b}=w^\epsilon+\Omega x$, where $w^\epsilon$ solves
\begin{eqnarray*}
\left\{\begin{array}{l}
\frac{\epsilon}{2}w^\epsilon_x=-\qty(w^\epsilon-(1-\beta-\Omega))\qty(w^\epsilon-\beta-\Omega),\\
w^\epsilon(0)=\alpha.
\end{array}\right.
\end{eqnarray*}
Then
\begin{eqnarray*}
&&L\rho^a=0,\\
&&\rho^a(1)=1-\beta,\\
&&\rho^{\epsilon,b}(0)=w^\epsilon(0)=\alpha.
\end{eqnarray*}

\subsubsection{$\alpha+\beta+\Omega<1$}
Because
\begin{eqnarray*}
&&\rho^a(0)=1-\beta-\Omega>\alpha,\\
&&\rho^{\epsilon,b}(1)=w^\epsilon(1)+\Omega=w^\epsilon(1)-1+\beta+\Omega+1-\beta<1-\beta,\\
&&L\rho^{\epsilon,b}=2\Omega xw^\epsilon_x\ge 0,
\end{eqnarray*}
we have $\forall\epsilon>0$, $\rho^a$ and $\rho^{\epsilon,b}$ satisfy the conditions of upper and lower solutions in Lemma \ref{upperlowersufficientcondition}.

Because $m\qty(\abs{\rho^a-\hat{\rho}}>0)=0$, $\forall\delta>0$, $\lim_{\epsilon\to 0}m\qty(\abs{\rho^{\epsilon,b}-\hat{\rho}}>\delta)=0$, $\rho^a>\rho^{\epsilon,b}$, by Theorem \ref{upperlower2W12}, we have $\forall\delta>0$, $\forall \epsilon>0$, Eq. \eqref{ellipticequationintroduction} has a $W^{1,2}(0,1)$ solution $\rho^\epsilon$,
\begin{eqnarray*}
\varlimsup_{\epsilon\to 0}m\qty(\abs{\rho^\epsilon-\hat{\rho}}> \delta)=0.
\end{eqnarray*}
Theorem \ref{TheEqOmega} follows Lemma \ref{LemConInMea}.

\subsubsection{$\alpha+\beta+\Omega>1$}

Because
\begin{eqnarray*}
&&\rho^a(0)=1-\beta-\Omega<\alpha,\\
&&\rho^{\epsilon,b}(1)=w^\epsilon(1)-1+\beta+\Omega+1-\beta>1-\beta,\\
&&L\rho^{\epsilon,b}=2\Omega xw^\epsilon_x\le0,
\end{eqnarray*}
we have $\forall\epsilon>0$, $\rho^a$ and $\rho^{\epsilon,b}$ satisfy the conditions of lower and upper solutions in Lemma \ref{upperlowersufficientcondition}. Theorem \ref{TheEqOmega} then follows Theorem \ref{upperlower2W12} and Lemma \ref{LemConInMea} similarly.

\subsection{Phase 5: $\alpha>0.5$, $\beta>0.5$}\label{spephase6seclab}
Define $\rho^{\epsilon,u}\in Y$ as $\rho^{\epsilon,u}\coloneqq w^\epsilon$, where $w^\epsilon$ solves
\begin{eqnarray*}
\left\{\begin{array}{l}
\frac{\epsilon}{2}w^\epsilon_x=-\qty(w^\epsilon-0.5)^2,\\
w^\epsilon(0)=\alpha.
\end{array}\right.
\end{eqnarray*}
Because
\begin{eqnarray*}
&&L\rho^{\epsilon,u}=-\Omega\qty(2w^\epsilon-1)<0,\\
&&\rho^{\epsilon,u}(0)=w^\epsilon(0)=\alpha,\\
&&\rho^{\epsilon,u}(1)=w^\epsilon(1)>0.5>1-\beta,
\end{eqnarray*}
we have $\forall\epsilon>0$, $\rho^{\epsilon,u}$ satisfies the conditions of upper solutions in Lemma \ref{upperlowersufficientcondition}. $\forall\delta>0$, $\lim_{\epsilon\to0}m\qty(\abs{\rho^{\epsilon,u}-\hat{\rho}}\ge\delta)=0$.

Define $\rho^{\epsilon,l}$ by symmetry. Because $\rho^{\epsilon,l}<\rho^{\epsilon,u}$, by Theorem \ref{upperlower2W12}, we have $\forall\delta>0$, $\forall \epsilon>0$, Eq. \eqref{ellipticequationintroduction} has a $W^{1,2}(0,1)$ solution $\rho^\epsilon$,
\begin{eqnarray*}
\varlimsup_{\epsilon\to 0}m\qty(\abs{\rho^\epsilon-\hat{\rho}}> \delta)=0.
\end{eqnarray*}
Theorem \ref{TheEqOmega} follows Lemma \ref{LemConInMea}.

\section{Phases of Eq. \eqref{ellipticequationintroduction} as $\epsilon\to 0$ for $\Omega_A/\Omega_D>1$}\label{phasegensec}
By particle-hole symmetry in Section \ref{model}, assume $K\equiv\Omega_A/\Omega_D>1$. Let $\epsilon=0$ in Eq. \eqref{ellipticequationintroduction}.
\begin{eqnarray}\label{Ldefinitiongenlim}
2(\rho-0.5)d\rho=(K+1)\Omega_D(\rho-\frac{K}{K+1})dx.
\end{eqnarray}
Curves of Eq. \eqref{Ldefinitiongenlim} is summarized in Fig. \ref{gene_phase}(a).
\begin{figure}
  \centering
  \includegraphics[width=14cm]{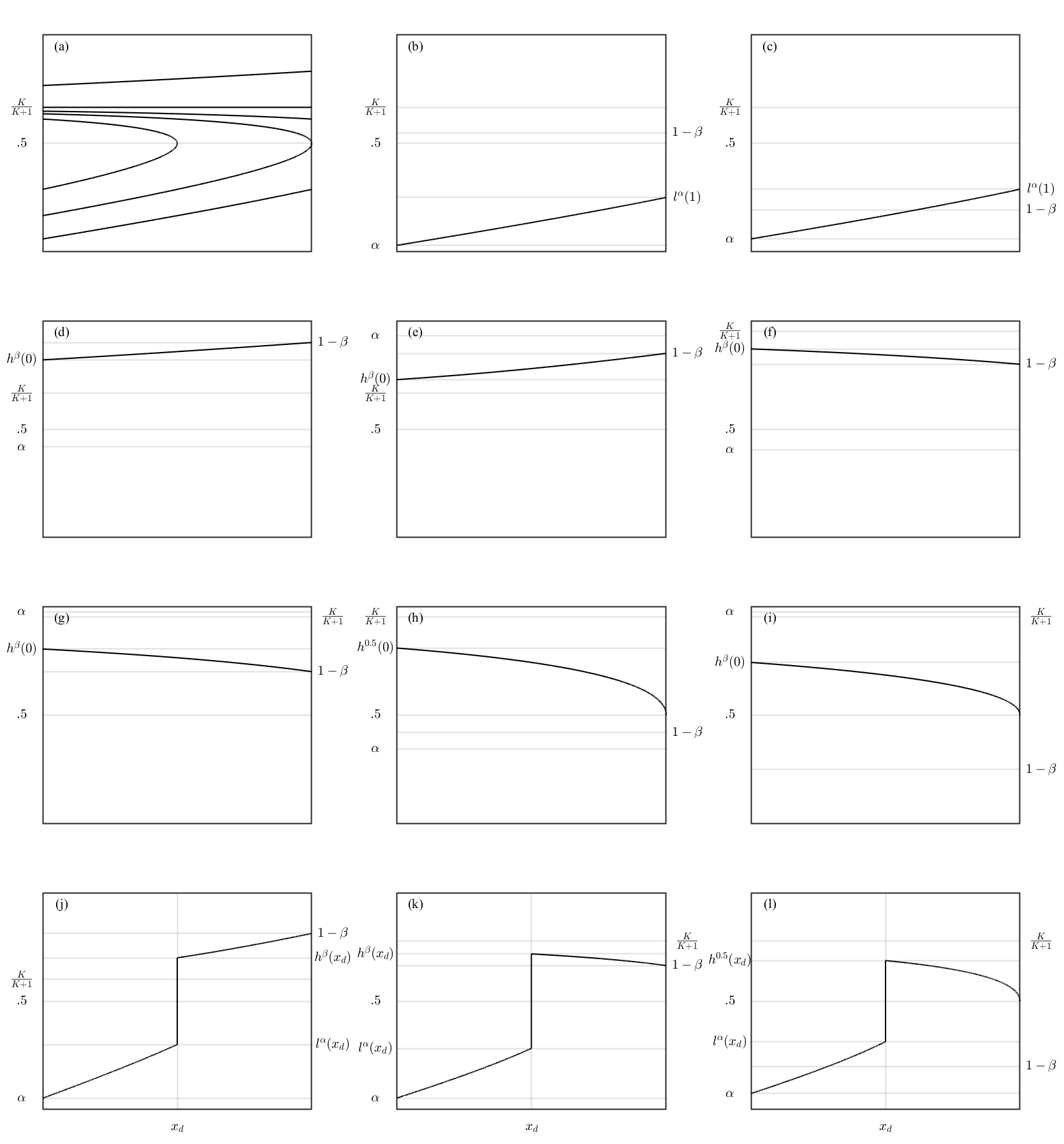}\\
  \caption{(a) Curves of Eq. \eqref{Ldefinitiongenlim}. (b-l) Diagrams of $\hat{\rho}$ for $K>1$. (b) $\Omega_D=0.1$, $K=2$, $\alpha=0.03$, $\beta=0.4502$. (c) $\Omega_D=0.1$, $K=2$, $\alpha=0.06$, $\beta=0.8074$. (d) $\Omega_D=0.1$, $K=2$, $\alpha=0.4200$, $\beta=0.1000$. (e) $\Omega_D=0.2$, $K=2$, $\alpha=0.9323$, $\beta=0.1500$. (f) $\Omega_D=0.02$, $K=20$, $\alpha=0.4074$, $\beta=0.2000$. (g) $\Omega_D=0.013$, $K=20$, $\alpha=0.9762$, $\beta=0.3000$. (h) $\Omega_D=0.02$, $K=20$, $\alpha=0.3457$, $\beta=0.5800$. (i) $\Omega_D=0.01$, $K=20$, $\alpha=0.9762$, $\beta=0.7500$. (j) $\Omega_D=0.3$, $K=1.5$, $x_d=0.5$. (k) $\Omega_D=0.11$, $K=3.5$, $x_d=0.5$. (l) $\Omega_D=0.11$, $K=3.5$, $x_d=0.5$.}\label{gene_phase}
\end{figure}

\begin{definition}
$\forall \alpha\in[0,0.5]$, let $l^\alpha$ be the part in $\rho\le 0.5$ of the curve passing $(0,\alpha)$ of Eq. \eqref{Ldefinitiongenlim}. $\forall \beta\in[0,0.5]$, let $h^\beta$ be the part in $\rho\ge 0.5$ of the curve passing $(1,1-\beta)$ of Eq. \eqref{Ldefinitiongenlim}; $g^\beta$ be the part in $\rho\le 0.5$ of the curve passing $(1,\beta)$ of Eq. \eqref{Ldefinitiongenlim}.
\end{definition}

\begin{lemma}\label{LemXdEU}
If $\beta\le 0.5$, $g^\beta(0)<\alpha<1-h^\beta(0)$, then $\exists x_d\in(0,1)$,
\begin{enumerate}
\item $l^\alpha(x_d)+h^\beta(x_d)=1$.
\item $\forall x\in[0,x_d)$, $l^\alpha+h^\beta<1$.
\item $\forall x\in(x_d,1]$, $l^\alpha+h^\beta>1$.
\end{enumerate}
\end{lemma}
\begin{proof}
By definition, $l^\alpha(0)+h^\beta(0)=\alpha+h^\beta(0)<1$. Prove by contradiction that $\exists x_d\in(0,1)$, $l^\alpha(x_d)+h^\beta(x_d)=1$. Otherwise, $\forall x\in[0,1)$, $l^\alpha+h^\beta< 1$, thereby $l^\alpha< 1-h^\beta\le 0.5$. Thus, $l^\alpha$ exists in $[0,1)$. Because $g^\beta(0)<\alpha$, we have $l^\alpha(1^-)+h^\beta(1^-)>\beta+1-\beta=1$. By continuity, $\exists x_d\in(0,1)$, $l^\alpha(x_d)+h^\beta(x_d)=1$, conflicts.

Now it is enough to prove that $\forall x\in[0,1]$, $l^\alpha+h^\beta=1$ implies $l_x^\alpha+h_x^\beta>0$.
\begin{enumerate}
\item If $\beta\le \frac{1}{K+1}$, then $h^\beta\ge\frac{K}{K+1}$. Because $l^\alpha< 0.5$, we have
\begin{eqnarray*}
&&l_x^\alpha+h_x^\beta\\
&&=\frac{(K+1)\Omega_D\qty(l^\alpha-\frac{K}{K+1})}{2(l^\alpha-0.5)}+\frac{(K+1)\Omega_D\qty(h^\beta-\frac{K}{K+1})}{2(h^\beta-0.5)}>0.
\end{eqnarray*}

\item If $\beta>\frac{1}{K+1}$, then $h^\beta<\frac{K}{K+1}$.
Because $0<0.5-l^\alpha=h^\beta-0.5$, we have
\begin{eqnarray*}
&&l_x^\alpha+h_x^\beta\\
&&=\frac{(K+1)\Omega_D\qty(l^\alpha-\frac{K}{K+1})}{2(l^\alpha-0.5)}+\frac{(K+1)\Omega_D\qty(h^\beta-\frac{K}{K+1})}{2(0.5-l^\alpha)}\\
&&=\frac{(K+1)\Omega_D\qty(h^\beta-l^\alpha)}{2(0.5-l^\alpha)}>0.
\end{eqnarray*}
\end{enumerate}
\end{proof}

\begin{lemma}\label{LemwPro}
$\forall A\in(-\infty,0.5]$, if
\begin{eqnarray*}
\left\{\begin{array}{l}
\frac{\epsilon}{2}w^\epsilon_x=-(w^\epsilon-A)(w^\epsilon-(1-A)),\\
w^\epsilon(x_0)=w_0,
\end{array}\right.
\end{eqnarray*}
then
\begin{enumerate}
\item $\frac{\epsilon}{2}w^\epsilon_{xx}+(2w^\epsilon-1)w^\epsilon_x=0$.
\item if $A\neq 0.5$, then
\begin{eqnarray*}
&&\log\abs{w^\epsilon-A}-\log\abs{w^\epsilon-(1-A)}+\frac{2}{\epsilon}(2A-1)x\\
&&=\log\abs{w_0-A}-\log\abs{w_0-(1-A)}+\frac{2}{\epsilon}(2A-1)x_0.
\end{eqnarray*}
\item if $x<x_0$ and $w_0<1-A$, then
\begin{eqnarray*}
&&\abs{w^\epsilon-A}\le\exp\bigg(\frac{2}{\epsilon}(2A-1)(x_0-x)+\log\abs{w_0-A}\\
&&-\log\abs{w_0-(1-A)}+\log(1-A-\min(A,w_0))\bigg).
\end{eqnarray*}
\item if $x>x_0$ and $w_0>A$, then
\begin{eqnarray*}
&&\abs{w^\epsilon-(1-A)}\le\exp\bigg(\frac{2}{\epsilon}(2A-1)(x-x_0)-\log\abs{w_0-A}\\
&&+\log\abs{w_0-(1-A)}+\log(\max(1-A,w_0)-A)\bigg).
\end{eqnarray*}
\end{enumerate}
\end{lemma}

There are 5 phases depending on $\alpha$, $\beta$, $\Omega_D$, $K$. 
\begin{enumerate}
\item If $\alpha<g^{0.5}(0)$, $\beta>l^\alpha(1)$, then (Fig. \ref{gene_phase}(a,b))
\begin{eqnarray*}
\hat{\rho}=\left\{\begin{array}{ll}
l^\alpha, & 0\le x< 1, \\
\overline{\beta}, & x=1.
\end{array}\right.
\end{eqnarray*}

\item If $\beta<0.5$, $1-h^\beta(0)<\alpha$, then (Fig. \ref{gene_phase}(c,d,e,f))
\begin{eqnarray*}
\hat{\rho}=\left\{\begin{array}{ll}
\alpha, & x=0, \\
h^\beta, & 0< x\le 1.
\end{array}\right.
\end{eqnarray*}

\item If $\beta>0.5$, $1-h^{0.5}(0)<\alpha$, then (Fig. \ref{gene_phase}(g,h))
\begin{eqnarray*}
\hat{\rho}=\left\{\begin{array}{ll}
\alpha, & x=0, \\
h^{0.5}, & 0< x< 1, \\
\overline{\beta}, & x=1.
\end{array}\right.
\end{eqnarray*}

\item If $\beta<0.5$, $g^\beta(0)<\alpha<1-h^\beta(0)$, then (Fig. \ref{gene_phase}(i,j))
\begin{eqnarray*}
\hat{\rho}=\left\{\begin{array}{ll}
l^\alpha, & 0\le x\le x_d, \\
h^\beta, & x_d< x\le 1.
\end{array}\right.
\end{eqnarray*}

\item If $\beta>0.5$, $g^{0.5}(0)<\alpha<1-h^{0.5}(0)$, then (Fig. \ref{gene_phase}(k))
\begin{eqnarray*}
\hat{\rho}=\left\{\begin{array}{ll}
l^\alpha, & 0\le x\le x_d, \\
h^{0.5}, & x_d< x\le 1.
\end{array}\right.
\end{eqnarray*}
\end{enumerate}

The phase diagram may change with $\Omega_A$ and $\Omega_D$. We show three typical phase diagrams in Fig. \ref{gene_space}.

\begin{figure}
  \centering
  \includegraphics[width=14cm]{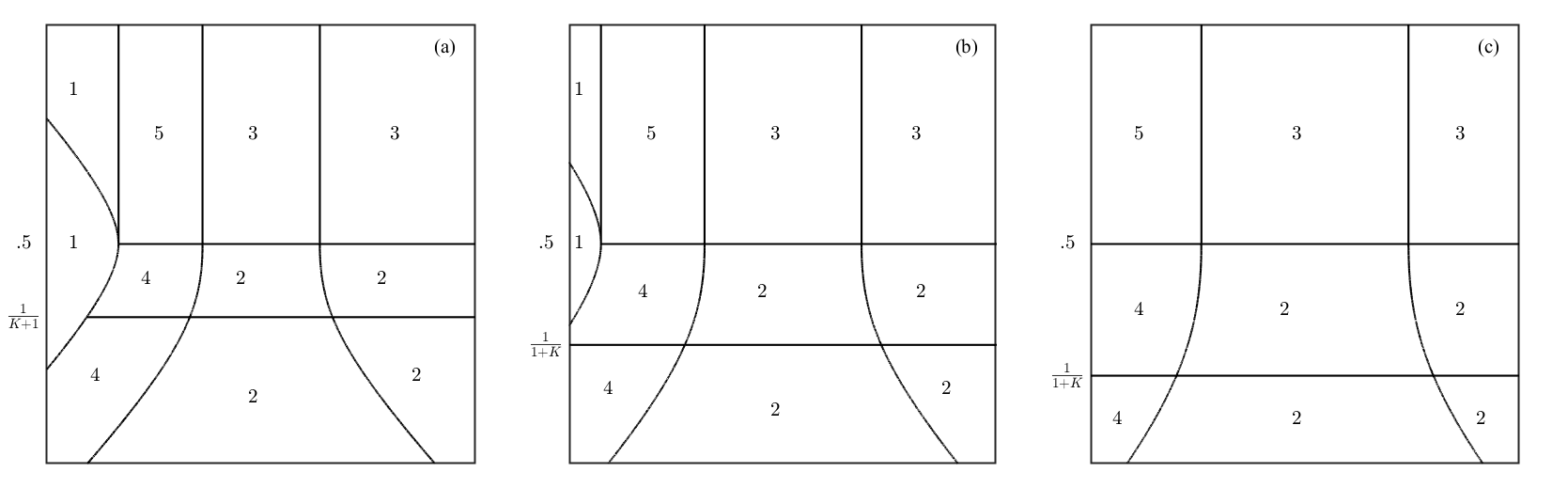}\\
  \caption{Phase diagrams of $\hat{\rho}$ for $K>1$. (a) $\Omega_D=0.1$, $K=2$. (b) $\Omega_D=0.1$, $K=2.7$. (c) $\Omega_D=0.1$, $K=4$.}\label{gene_space}
\end{figure}

\begin{theorem}\label{TheGeOmega}
If $\Omega_A/\Omega_D>1$, then $\exists\epsilon_0>0$, $\forall\epsilon<\epsilon_0$, Eq. \eqref{ellipticequationintroduction} has a $w^{1,2}(0,1)$ solution $\rho^\epsilon$, which converges to $\hat{\rho}$ in (Lebesgue) measure.
\end{theorem}

\section{Proof of Theorem \ref{TheGeOmega}}\label{steadyGeneralCases}
If
\begin{eqnarray*}
&&\rho_x=\frac{(K+1)\Omega_D(\rho-\frac{K}{K+1})}{2(\rho-0.5)},\\
&&\frac{\epsilon}{2}w^{\epsilon,1}_x=-(1-u)\qty(w^{\epsilon,1}-A_1)\qty(w^{\epsilon,1}-1+A_1),\\
&&\frac{\epsilon}{2}w^{\epsilon,2}_x=-\qty(w^{\epsilon,2}-A_2)\qty(w^{\epsilon,2}-1+A_2),
\end{eqnarray*}
then
\begin{eqnarray*}
&&L\qty(\rho+w^{\epsilon,1}+w^{\epsilon,2}-A_1-A_2-A_3)=\frac{\epsilon}{2}\rho_{xx}-A_3\Omega_D(K+1)\frac{0.5-\frac{K}{K+1}}{\rho-0.5}\\
&&+\Bigg[\Omega_D(K+1)\frac{0.5-\frac{K}{K+1}}{\rho-0.5}\\
&&-\frac{4}{\epsilon}\qty(\rho+w^{\epsilon,2}-A_1-A_2-A_3+u\qty(w^{\epsilon,1}-0.5))\qty(w^{\epsilon,1}-1+A_1)\Bigg]\qty(w^{\epsilon,1}-A_1)\\
&&+\Bigg[\Omega_D(K+1)\frac{0.5-\frac{K}{K+1}}{\rho-0.5}\\
&&-\frac{4}{\epsilon}\qty(\rho+w^{\epsilon,1}-A_1-A_2-A_3)\qty(w^{\epsilon,2}-1+A_2)\Bigg]\qty(w^{\epsilon,2}-A_2).
\end{eqnarray*}

\subsection{Phase 1: $\alpha<g^{0.5}(0)$, $\beta>l^\alpha(1)$}\label{genphase1seclab}

\subsubsection{$\beta<1-l^\alpha(1)$}

Let $u=0$, $0<\delta<\min\qty(\frac{\min(\beta,1-\beta)-l^\alpha(1)}{2},0.5-l^\alpha(1))$.

\paragraph{Upper solution}
Let $A_1=l^\alpha(1)+\delta$, $w^{\epsilon,1}(1)=1-\beta+\delta$. $\exists\delta_\alpha>0$, $\forall x\in[0,1]$, $l^{\alpha+\delta_\alpha}-l^\alpha<\delta$. $\exists A_3>0$, $\forall x\in[0,1]$, $l^{\alpha+\delta_\alpha}-A_3>l^\alpha$. Let $\rho^{\epsilon,u}=l^{\alpha+\delta_\alpha}+w^{\epsilon,1}-A_1-A_3$.

Because
\begin{eqnarray*}
&&\min_{x\in[0,1]}\qty(l^{\alpha+\delta_\alpha}-A_1-A_3)\qty(w^{\epsilon,1}-1+A_1)\\
&&\ge\qty(l^{\alpha+\delta_\alpha}(1)-l^\alpha(1)-\delta-A_3)\qty(l^\alpha(1)+2\delta-\beta)>0,
\end{eqnarray*}
we have
\begin{eqnarray*}
&&\varlimsup_{\epsilon\to0}\max_{x\in[0,1]}L\rho^{\epsilon,u}\le-A_3\Omega_D(K+1)\frac{0.5-\frac{K}{K+1}}{\alpha+\delta_\alpha-0.5}<0.
\end{eqnarray*}
Moreover,
\begin{eqnarray*}
&&\rho^{\epsilon,u}(0)=\alpha+\delta_\alpha-A_3+w^{\epsilon,1}(0)-A_1>\alpha,\\
&&\rho^{\epsilon,u}(1)=1-\beta-l^\alpha(1)+l^{\alpha+\delta_\alpha}-A_3>1-\beta,\\
&&\varlimsup_{\epsilon\to0}m\qty(\abs{\rho^{\epsilon,u}-\hat{\rho}}>2\delta)\\
&&\le m\qty(\abs{l^{\alpha+\delta_\alpha}-A_3-l^\alpha}>\delta)+\varlimsup_{\epsilon\to0}m\qty(\abs{w^{\epsilon,1}-A_1}>\delta)=0.
\end{eqnarray*}

\paragraph{Lower solution}
Let $\rho^{\epsilon,l}=l^\alpha$. Then
\begin{eqnarray*}
&&\min_{x\in[0,1]}L\rho^{\epsilon,l}=\frac{\epsilon}{2}\min_{x\in[0,1]}l^\alpha_{xx}>0,\\
&&\rho^{\epsilon,l}(0)=l^\alpha(0)=\alpha,\\
&&\rho^{\epsilon,l}(1)=l^\alpha(1)<1-\beta,\\
&&m\qty(\abs{\rho^{\epsilon,l}-\hat{\rho}}>0)=m\qty(\abs{l^\alpha-l^\alpha}>0)=0.
\end{eqnarray*}

\paragraph{Summary}
Because $\rho^{\epsilon,l}<\rho^{\epsilon,u}$, by Theorem \ref{upperlower2W12}, we have, $\exists\epsilon_0>0$, $\forall \epsilon<\epsilon_0$, Eq. \eqref{ellipticequationintroduction} has a $W^{1,2}(0,1)$ solution $\rho^\epsilon$,
\begin{eqnarray*}
\varlimsup_{\epsilon\to 0}m\qty(\abs{\rho^\epsilon-\hat{\rho}}>2\delta)=0.
\end{eqnarray*}
Theorem \ref{TheEqOmega} follows Lemma \ref{LemConInMea}.

\subsubsection{$\beta>1-l^\alpha(1)$}

Let $u=0$, $0<\delta<0.5-l^\alpha(1)$.

\paragraph{Upper solution}

$\exists\delta_a>0$, $\forall x\in[0,1]$, $l^{\alpha+\delta_\alpha}-l^\alpha<\delta$. $\exists A_3>0$, $\forall x\in[0,1]$, $l^{\alpha+\delta_\alpha}-A_3>l^\alpha$. Let $\rho^{\epsilon,u}=l^{\alpha+\delta_\alpha}-A_3$. Then
\begin{eqnarray*}
&&\varlimsup_{\epsilon\to0}\max_{x\in[0,1]}L\rho^{\epsilon,u}=-A_3\Omega_D(K+1)\frac{0.5-\frac{K}{K+1}}{\alpha+\delta\alpha-0.5}<0,\\
&&\rho^{\epsilon,u}(0)=\alpha+\delta_\alpha-A_3>\alpha,\\
&&\rho^{\epsilon,u}(1)=l^{\alpha+\delta_\alpha}(1)-A_3>l^\alpha(1)>1-\beta,\\
&&m\qty(\abs{\rho^{\epsilon,u}-\hat{\rho}}>\delta)=m\qty(\abs{l^{\alpha+\delta_\alpha}-A_3-l^\alpha}>\delta)=0.
\end{eqnarray*}

\paragraph{Lower solution}

Let $A_1=l^\alpha(1)+\delta$, $w^{\epsilon,1}(1)=1-\beta+\delta$, $\rho^{\epsilon,l}=l^\alpha+w^{\epsilon,1}-A_1$.

Because
\begin{eqnarray*}
&&\min_{x\in[0,1]}\qty(l^\alpha-A_1)\qty(w^{\epsilon,1}-1+A_1)\ge\delta\qty(1-2l^\alpha(1)-2\delta)>0,\\
&&\min_{x\in[0,1]}l^\alpha_{xx}>0,
\end{eqnarray*}
we have, $\exists\epsilon_0>0$, $\forall\epsilon<\epsilon_0$,
\begin{eqnarray*}
\min_{x\in[0,1]}L\rho^{\epsilon,l}>0.
\end{eqnarray*}
Moreover,
\begin{eqnarray*}
&&\rho^{\epsilon,l}(0)=\alpha+w^{\epsilon,1}(0)-A_1<\alpha,\\
&&\rho^{\epsilon,l}(1)=1-\beta,\\
&&\varlimsup_{\epsilon\to0}m\qty(\abs{\rho^{\epsilon,l}-\hat{\rho}}>\delta)\le m\qty(\abs{l^\alpha-l^\alpha}>0)+\varlimsup_{\epsilon\to0}m\qty(\abs{w^{\epsilon,1}-A_1}>\delta)=0.
\end{eqnarray*}

\paragraph{Summary}

Because $\rho^{\epsilon,l}<\rho^{\epsilon,u}$, by Theorem \ref{upperlower2W12}, we have, $\exists\epsilon_0>0$, $\forall \epsilon<\epsilon_0$, Eq. \eqref{ellipticequationintroduction} has a $W^{1,2}(0,1)$ solution $\rho^\epsilon$,
\begin{eqnarray*}
\varlimsup_{\epsilon\to 0}m\qty(\abs{\rho^\epsilon-\hat{\rho}}>\delta)=0.
\end{eqnarray*}
Theorem \ref{TheEqOmega} follows Lemma \ref{LemConInMea}.

\subsection{Phase 2: $\beta<0.5$, $1-h^\beta(0)<\alpha$}

\subsubsection{$\alpha<h^\beta(0)$}
Let $u=0$, $0<\delta<\min\qty(\frac{h^\beta(0)-\max(\alpha,1-\alpha)}{2},h^\beta(0)-0.5)$.

\paragraph{Upper solution}
Let $-\delta<A_3<0$, $\rho^{\epsilon,u}=h^\beta-A_3$. Then
\begin{eqnarray*}
&&\varlimsup_{\epsilon\to0}\max_{x\in[0,1]}L\rho^{\epsilon,u}=-A_3\Omega_D(K+1)\frac{0.5-\frac{K}{K+1}}{\max_{x\in[0,1]}h^\beta-0.5}<0,\\
&&\rho^{\epsilon,u}(0)=h^\beta(0)-A_3>\alpha,\\
&&\rho^{\epsilon,u}(1)=1-\beta-A_3>1-\beta,\\
&&m\qty(\abs{\rho^{\epsilon,u}-\hat{\rho}}>\delta)=m\qty(\abs{h^\beta-A_3-h^\beta}>\delta)=0.
\end{eqnarray*}

\paragraph{Lower solution}
Let $A_1=h^\beta(0)-\delta$, $w^{\epsilon,1}(0)=\alpha-\delta$, $0<A_3<\delta$, $\rho^{\epsilon,l}=h^\beta+w^{\epsilon,1}-A_1-A_3$.

By continuity, $\exists0<\delta_0<1$,
\begin{eqnarray*}
&&\min_{x\in[0,\delta_0]}\qty(h^\beta-A_1-A_3)\qty(w^{\epsilon,1}-1+A_1)\\
&&\ge\qty(\min_{x\in[0,\delta_0]}h^\beta-h^\beta(0)+\delta-A_3)\qty(\alpha-2\delta-1+h^\beta(0))>0.
\end{eqnarray*}
By Lemma \ref{LemwPro},
\begin{eqnarray*}
&&\max_{x\in[\delta_0,1]}\abs{w^{\epsilon,1}-A_1}\le\exp\bigg(\frac{2}{\epsilon}(1-2h^\beta(0)+2\delta)\delta_0-\log\abs{\alpha-2\delta-1+h^\beta(0)}\\
&&+\log\abs{\alpha-h^\beta(0)}+\log(2h^\beta(0)-2\delta-1)\bigg).
\end{eqnarray*}
Therefore,
\begin{eqnarray*}
\varliminf_{\epsilon\to0}\min_{x\in[0,1]}L\rho^{\epsilon,l}\ge-A_3\Omega_D(K+1)\frac{0.5-\frac{K}{K+1}}{\max_{x\in[0,1]}h^\beta-0.5}>0.
\end{eqnarray*}
Moreover,
\begin{eqnarray*}
&&\rho^{\epsilon,l}(0)=\alpha-A_3<\alpha,\\
&&\rho^{\epsilon,l}(1)=1-\beta+w^{\epsilon,1}(1)-A_1-A_3<1-\beta,\\
&&\varlimsup_{\epsilon\to0}m\qty(\abs{\rho^{\epsilon,l}-\hat{\rho}}>2\delta)\\
&&\le m\qty(\abs{h^\beta-A_3-h^\beta}>\delta)+\varlimsup_{\epsilon\to0}m\qty(\abs{w^{\epsilon,1}-A_1}>\delta)=0.
\end{eqnarray*}

\paragraph{Summary}

Because $\rho^{\epsilon,l}<\rho^{\epsilon,u}$, by Theorem \ref{upperlower2W12}, we have, $\exists\epsilon_0>0$, $\forall \epsilon<\epsilon_0$, Eq. \eqref{ellipticequationintroduction} has a $W^{1,2}(0,1)$ solution $\rho^\epsilon$,
\begin{eqnarray*}
\varlimsup_{\epsilon\to 0}m\qty(\abs{\rho^\epsilon-\hat{\rho}}>2\delta)=0.
\end{eqnarray*}
Theorem \ref{TheEqOmega} follows Lemma \ref{LemConInMea}.

\subsubsection{$\alpha>h^\beta(0)$}

Let $u=0$, $0<\delta<h^\beta(0)-0.5$.

\paragraph{Upper solution}

Let $A_1=h^\beta(0)-\delta$, $w^{\epsilon,1}(0)=\alpha-\delta$, $-\delta<A_3<0$, $\rho^{\epsilon,u}=h^\beta+w^{\epsilon,1}-A_1-A_3$. 

By continuity, $\exists 0<\delta_0<1$,
\begin{eqnarray*}
&&\min_{x\in[0,\delta_0]}\qty(h^\beta-A_1-A_3)\qty(w^{\epsilon,1}-1+A_1)\\
&&\ge\qty(\min_{x\in[0,\delta_0]}h^\beta-h^\beta(0)+\delta-A_3)\qty(2h^\beta(0)-2\delta-1)>0.
\end{eqnarray*}
By Lemma \ref{LemwPro},
\begin{eqnarray*}
\max_{x\in[\delta_0,1]}\abs{w^{\epsilon,1}-A_1}\le\exp\bigg(\frac{2}{\epsilon}(1-2h^\beta(0)+2\delta)\delta_0+\log\abs{\alpha-h^\beta(0)}\bigg).
\end{eqnarray*}
Therefore,
\begin{eqnarray*}
\varlimsup_{\epsilon\to0}\max_{x\in[0,1]}L\rho^{\epsilon,u}\le-A_3\Omega_D(K+1)\frac{0.5-\frac{K}{K+1}}{\max_{x\in[0,1]}h^\beta-0.5}<0.
\end{eqnarray*}
Moreover,
\begin{eqnarray*}
&&\rho^{\epsilon,u}(0)=\alpha-A_3>\alpha,\\
&&\rho^{\epsilon,u}1=1-\beta+w^{\epsilon,1}-A_1-A_3>1-\beta,\\
&&\varlimsup_{\epsilon\to0}m\qty(\abs{\rho^{\epsilon,u}-\hat{\rho}}>2\delta)\\
&&\le m\qty(\abs{h^\beta-A_3-h^\beta}>\delta)+\varlimsup_{\epsilon\to0}m\qty(\abs{w^{\epsilon,1}-A_1}>\delta)=0.
\end{eqnarray*}

\paragraph{Lower solution}

Let $0<A_3<\delta$, $\rho^{\epsilon,l}=h^\beta-A_3$. Then
\begin{eqnarray*}
&&\varliminf_{\epsilon\to0}\min_{x\in[0,1]}L\rho^{\epsilon,l}=-A_3\Omega_D(K+1)\frac{0.5-\frac{K}{K+1}}{\max_{x\in[0,1]}h^\beta-0.5}>0,\\
&&\rho^{\epsilon,l}(0)=h^\beta(0)-A_3<\alpha,\\
&&\rho^{\epsilon,l}(1)=1-\beta-A_3<1-\beta,\\
&&m\qty(\abs{\rho^{\epsilon,l}-\hat{\rho}}>\delta)=m\qty(\abs{h^\beta-A_3-h^\beta}>\delta)=0.
\end{eqnarray*}

\paragraph{Summary}

Because $\rho^{\epsilon,l}<\rho^{\epsilon,u}$, by Theorem \ref{upperlower2W12}, we have, $\exists\epsilon_0>0$, $\forall \epsilon<\epsilon_0$, Eq. \eqref{ellipticequationintroduction} has a $W^{1,2}(0,1)$ solution $\rho^\epsilon$,
\begin{eqnarray*}
\varlimsup_{\epsilon\to 0}m\qty(\abs{\rho^\epsilon-\hat{\rho}}>2\delta)=0.
\end{eqnarray*}
Theorem \ref{TheEqOmega} follows Lemma \ref{LemConInMea}.

\subsection{Phase 3: $\beta>0.5$, $1-h^{0.5}(0)<\alpha$}

\subsubsection{$\alpha<h^{0.5}(0)$}
Let $u=0$, $0<\delta<\min\qty(\frac{h^{0.5}(0)-\max(\alpha,1-\alpha)}{2},h^{0.5}(0)-0.5)$.

\paragraph{Upper solution}

$\exists\delta_\beta>0$, $\forall x\in[0,1]$, $h^{0.5-\delta_\beta}-h^{0.5}<\delta$. Let $\rho^{\epsilon,u}=h^{0.5-\delta_\beta}$. Then
\begin{eqnarray*}
&&\max_{x\in[0,1]}L\rho^{\epsilon,u}=\frac{\epsilon}{2}\max_{x\in[0,1]}h^{0.5-\delta_\beta}_{xx}<0,\\
&&\rho^{\epsilon,u}(0)=h^{0.5-\delta_\beta}(0)>h^{0.5}(0)>\alpha,\\
&&\rho^{\epsilon,u}(1)=0.5+\delta_\beta>0.5>1-\beta,\\
&&m\qty(\abs{\rho^{\epsilon,u}-\hat{\rho}}>\delta)=m\qty(\abs{h^{0.5-\delta_\beta}-h^{0.5}}>\delta)=0.
\end{eqnarray*}

\paragraph{Lower solution}
Let $A_1=h^{0.5}(0)-\delta$, $w^{\epsilon,1}(0)=\alpha-\delta$. $\exists\delta_\beta>0$, $\forall x\in[0,1]$, $h^{0.5-\delta_\beta}-h^{0.5}<\delta$. Let $0<\delta_2<\delta-\delta_\beta$, $A_2=0.5-\delta_2$, $w^{\epsilon,2}(1)=1-\beta-\delta_2$, $A_3=\delta$, $\rho^{\epsilon,l}=h^{0.5-\delta_\beta}+w^{\epsilon,1}+w^{\epsilon,2}-A_1-A_2-A_3$.

By continuity, $\exists0<\delta_0<1$,
\begin{eqnarray*}
&&\varliminf_{\epsilon\to0}\min_{x\in[0,\delta_0]}\qty(h^{0.5-\delta_\beta}+w^{\epsilon,2}-A_1-A_2-A_3)\qty(w^{\epsilon,1}-1+A_1)\\
&&\ge\qty(h^{0.5-\delta_\beta}(\delta_0)-h^{0.5}(0))\qty(\alpha-2\delta-1+h^{0.5}(0))>0.
\end{eqnarray*}
By Lemma \ref{LemwPro},
\begin{eqnarray*}
&&\max_{x\in[\delta_0,1]}\abs{w^{\epsilon,1}-A_1}\le\exp\bigg(\frac{2}{\epsilon}\qty(1+2\delta-2h^{0.5}(0))\delta_0\\
&&-\log\abs{\alpha-2\delta-1+h^{0.5}(0)}+\log\abs{\alpha-h^{0.5}(0)}\\
&&+\log(2h^{0.5}(0)-2\delta-1)\bigg).
\end{eqnarray*}
By continuity, $\exists0<\delta_1<1$,
\begin{eqnarray*}
&&\varliminf_{\epsilon\to0}\min_{x\in[1-\delta_1,1]}\qty(h^{0.5-\delta_\beta}+w^{\epsilon,1}-A_1-A_2-A_3)\qty(w^{\epsilon,2}-1+A_2)\\
&&\ge2\delta_2\qty(0.5+\delta-h^{0.5-\delta_\beta}(1-\delta_1)-\delta_2)>0.
\end{eqnarray*}
By Lemma \ref{LemwPro},
\begin{eqnarray}\label{EqGenPhase3w2}
\max_{x\in[0,1-\delta_1]}\abs{w^{\epsilon,2}-A_2}\le\exp\bigg(-\frac{4\delta_2\delta_1}{\epsilon}+\log\abs{0.5-\beta}\bigg).
\end{eqnarray}
Therefore,
\begin{eqnarray*}
\varliminf_{\epsilon\to0}\min_{x\in[0,1]}L\rho^{\epsilon,l}\ge-\delta\Omega_D(K+1)\frac{0.5-\frac{K}{K+1}}{h^{0.5-\delta_\beta}(0)-0.5}>0.
\end{eqnarray*}
Moreover,
\begin{eqnarray*}
&&\rho^{\epsilon,l}(0)=\alpha-\delta+h^{0.5-\delta_\beta}(0)-h^{0.5}(0)+w^{\epsilon,2}(0)-0.5+\delta_2<\alpha,\\
&&\rho^{\epsilon,l}(1)=1-\beta+\delta_\beta+w^{\epsilon,1}(1)-h^{0.5}(0)<1-\beta,\\
&&\varlimsup_{\epsilon\to0}m\qty(\abs{\rho^{\epsilon,l}-\hat{\rho}}>2\delta)\\
&&\le m\qty(\abs{h^{0.5-\delta_\beta}-\delta-h^{0.5}}>\delta)+\varlimsup_{\epsilon\to0}m\qty(\abs{w^{\epsilon,1}-A_1+w^{\epsilon,2}-A_2}>\delta)=0.
\end{eqnarray*}

\paragraph{Summary}

Because $\rho^{\epsilon,l}<\rho^{\epsilon,u}$, by Theorem \ref{upperlower2W12}, we have, $\exists\epsilon_0>0$, $\forall \epsilon<\epsilon_0$, Eq. \eqref{ellipticequationintroduction} has a $W^{1,2}(0,1)$ solution $\rho^\epsilon$,
\begin{eqnarray*}
\varlimsup_{\epsilon\to 0}m\qty(\abs{\rho^\epsilon-\hat{\rho}}>2\delta)=0.
\end{eqnarray*}
Theorem \ref{TheEqOmega} follows Lemma \ref{LemConInMea}.

\subsubsection{$\alpha>h^{0.5}(0)$}

Let $u=0$, $0<\delta<h^{0.5}(0)-0.5$.

\paragraph{Upper solution}

Let $A_1=h^{0.5}(0)$, $w^{\epsilon,1}(0)=\alpha$. $\exists\delta_\beta>0$, $\forall x\in[0,1]$, $h^{0.5-\delta_\beta}-h^{0.5}<\delta$. $\rho^{\epsilon,u}=h^{0.5-\delta_\beta}+w^{\epsilon,1}-A_1$.

By continuity, $\exists 0<\delta_0<1$,
\begin{eqnarray*}
&&\min_{x\in[0,\delta_0]}\qty(h^{0.5-\delta_\beta}-A_1)\qty(w^{\epsilon,1}-1+A_1)\\
&&\ge\qty(h^{0.5-\delta_\beta}(\delta_0)-h^{0.5}(0))\qty(2h^{0.5}(0)-1)>0.
\end{eqnarray*}
By Lemma \ref{LemwPro},
\begin{eqnarray*}
\max_{x\in[\delta_0,1]}\abs{w^{\epsilon,1}-A_1}\le\exp\bigg(\frac{2}{\epsilon}\qty(1-2h^{0.5}(0))\delta_0+\log\abs{\alpha-h^{0.5}(0)}\bigg).
\end{eqnarray*}
Because $\max_{x\in[0,1]}h^{0.5}_{xx}<0$, we have, $\exists\epsilon_0>0$, $\forall\epsilon<\epsilon_0$,
\begin{eqnarray*}
\max_{x\in[0,1]}L\rho^{\epsilon,u}<0.
\end{eqnarray*}
Moreover,
\begin{eqnarray*}
&&\rho^{\epsilon,u}(0)=h^{0.5-\delta_\beta}(0)+\alpha-h^{0.5}(0)>\alpha,\\
&&\rho^{\epsilon,u}(1)=0.5+\delta_\beta+w^{\epsilon,1}(1)-A_1>0.5>1-\beta,\\
&&\varlimsup_{\epsilon\to 0}m\qty(\abs{\rho^{\epsilon,u}-\hat{\rho}}>2\delta)\\
&&\le m\qty(\abs{h^{0.5-\delta_\beta}-h^{0.5}}>\delta)+\varlimsup_{\epsilon\to 0}m\qty(\abs{w^{\epsilon,1}-A_1}>\delta)=0.
\end{eqnarray*}

\paragraph{Lower solution}

$\exists\delta_\beta>0$, $\forall x\in[0,1]$, $h^{0.5-\delta_\beta}-h^{0.5}<\delta$. Let $0<\delta_2<\delta-\delta_\beta$, $A_2=0.5-\delta_2$, $w^{\epsilon,2}(1)=1-\beta-\delta_2$, $A_3=\delta$, $\rho^{\epsilon,l}=h^{0.5-\delta_\beta}+w^{\epsilon,2}-A_2-A_3$.

By continuity, $\exists0<\delta_1<1$,
\begin{eqnarray*}
&&\min_{x\in[1-\delta_1,1]}\qty(h^{0.5-\delta_\beta}-A_2-A_3)\qty(w^{\epsilon,2}-1+A_2)\\
&&\ge2\delta_2\qty(0.5+\delta-h^{0.5-\delta_\beta}(1-\delta_1)-\delta_2)>0.
\end{eqnarray*}
By Lemma \ref{LemwPro}, we have Eq. \eqref{EqGenPhase3w2}. Therefore,
\begin{eqnarray*}
\varliminf_{\epsilon\to0}\min_{x\in[0,1]}L\rho^{\epsilon,l}\le-\delta\Omega_D(K+1)\frac{0.5-\frac{K}{K+1}}{h^{0.5-\delta_\beta}(0)-0.5}>0.
\end{eqnarray*}
Moreover,
\begin{eqnarray*}
&&\rho^{\epsilon,l}(0)=h^{0.5-\delta_\beta}(0)-\delta+w^{\epsilon,2}(0)-0.5+\delta_2<h^{0.5}(0)<\alpha,\\
&&\rho^{\epsilon,l}(1)=1-\beta+\delta_\beta-\delta<1-\beta,\\
&&\varlimsup_{\epsilon\to0}m\qty(\abs{\rho^{\epsilon,l}-\hat{\rho}}>2\delta)\\
&&\le m\qty(\abs{h^{0.5}-\delta-h^{0.5}}>\delta)+\varlimsup_{\epsilon\to0}m\qty(\abs{w^{\epsilon,2}-A_2}>\delta)=0.
\end{eqnarray*}

\paragraph{Summary}

Because $\rho^{\epsilon,l}<\rho^{\epsilon,u}$, by Theorem \ref{upperlower2W12}, we have, $\exists\epsilon_0>0$, $\forall \epsilon<\epsilon_0$, Eq. \eqref{ellipticequationintroduction} has a $W^{1,2}(0,1)$ solution $\rho^\epsilon$,
\begin{eqnarray*}
\varlimsup_{\epsilon\to 0}m\qty(\abs{\rho^\epsilon-\hat{\rho}}>2\delta)=0.
\end{eqnarray*}
Theorem \ref{TheEqOmega} follows Lemma \ref{LemConInMea}.

\subsection{Phase 4: $\beta<0.5$, $g^\beta(0)<\alpha<1-h^\beta(0)$}

Let $0<\delta<\min\qty(1-h^\beta(0)-\alpha,\alpha-g^{\beta}(0))$.

\paragraph{Upper solution}

$\exists \delta_0>0$, $\forall x\in[0,1]$, $l^{\alpha+\delta_\alpha}-l^\alpha<\delta$. Then $\alpha+\delta_\alpha<\alpha+\delta<1-h^\beta(0)$. By Lemma \ref{LemXdEU}, $\exists x_{d,u}\in(0,1)$, $l^{\alpha+\delta_\alpha}(x_{d,u})+h^\beta(x_{d,u})=1$.
Because
\begin{eqnarray*}
l^{\alpha+\delta_\alpha}(x_d)+h^\beta(x_d)>l^\alpha,(x_d)+h^\beta(x_d)=1,
\end{eqnarray*}
by Lemma \ref{LemXdEU}, $x_{d,u}<x_d$. $\exists\delta_0>0$, $\forall x\in[0,1]$, $l^{\alpha+\delta_\alpha}-\delta_0>l^\alpha$. Let $u=0$, $w^{\epsilon,1}(x_{d,u})=0.5$,
\begin{eqnarray*}
&&A_1=\left\{\begin{array}{ll}
l^{\alpha+\delta_\alpha}(x_{d,u})-\frac{2}{3}\delta_0, & 0\le x\le x_{d,u},\\
h^\beta(x_{d,u})+\frac{2}{3}\delta_0, & x_{d,u}<x\le 1,
\end{array}\right.\\
&&A_3=\left\{\begin{array}{ll}
\delta_0, & 0\le x\le x_{d,u},\\
-\frac{\delta_0}{3}, & x_{d,u}<x\le 1,
\end{array}\right.\\
&&\rho^u=\left\{\begin{array}{ll}
l^{\alpha+\delta_\alpha}, & 0\le x\le x_{d,u},\\
h^\beta, & x_{d,u}<x\le 1.
\end{array}\right.\\
&&\rho^{\epsilon,u}=\rho^u+w^{\epsilon,1}-A_1-A_3.
\end{eqnarray*}
Because
\begin{eqnarray*}
&&\min_{x\in[0,x_{d,u}]}\qty(l^{\alpha+\delta_\alpha}-A_1-A_3)\qty(w^{\epsilon,1}-1+A_1)\\
&&\ge\frac{\delta_0}{3}\qty(0.5-l^{\alpha+\delta_\alpha}(x_{d,u})+\frac{2}{3}\delta_0)>0,
\end{eqnarray*}
we have
\begin{eqnarray*}
\varlimsup_{\epsilon\to0}\max_{x\in[0,x_{d,u}]}L\rho^{\epsilon,u}\le-\delta_0\Omega_D(K+1)\frac{0.5-\frac{K}{K+1}}{\alpha+\delta_\alpha-0.5}<0.
\end{eqnarray*}
By continuity, $\exists0<\delta_h<1-x_{d,u}$,
\begin{eqnarray*}
&&\sup_{x\in(x_{d,u},x_{d,u}+\delta_h]}\qty(h^\beta-A_1-A_3)\qty(w^{\epsilon,1}-1+A_1)\\
&&\le\qty(\max_{x\in[x_{d,u},x_{d,u}+\delta_h]}h^\beta-h^\beta(x_{d,u})-\frac{\delta_0}{3})\qty(0.5-l^{\alpha+\delta_\alpha}(x_{d,u})+\frac{2}{3}\delta_0)<0.
\end{eqnarray*}
By Lemma \ref{LemwPro},
\begin{eqnarray*}
&&\max_{x\in[x_{d,u}+\delta_h,1]}\abs{w^{\epsilon,1}-A_1}\\
&&\le\exp\bigg(\frac{2}{\epsilon}(1-2h^\beta(x_{d,u})-\frac{4}{3}\delta_0)\delta_h-\log\abs{0.5-l^{\alpha+\delta_\alpha}(x_{d,u})+\frac{2}{3}\delta_0}\\
&&+\log\abs{0.5-h^\beta(x_{d,u})-\frac{2}{3}\delta_0}+\log(2h^\beta(x_{d,u})+\frac{4}{3}\delta_0-1)\bigg).
\end{eqnarray*}
Therefore,
\begin{eqnarray*}
\varlimsup_{\epsilon\to0}\sup_{x\in(x_{d,u},1]}L\rho^{\epsilon,u}\le\frac{\delta_0}{3}\Omega_D(K+1)\frac{0.5-\frac{K}{K+1}}{\max_{x\in[x_{d,u},1]}h^\beta-0.5}<0.
\end{eqnarray*}
Moreover,
\begin{eqnarray*}
&&\rho^{\epsilon,u}_x\qty(x_{d,u}^-)=\frac{(K+1)\Omega_D\qty(l^{\alpha+\delta_\alpha}(x_{d,u})-\frac{K}{K+1})}{2(l^{\alpha+\delta_\alpha}(x_{d,u})-0.5)}+w^{\epsilon,1}_x(x_{d,u})\\
&&>\frac{(K+1)\Omega_D}{2}+w^{\epsilon,1}_x(x_{d,u})\\
&&>\frac{(K+1)\Omega_D\qty(h^\beta(x_{d,u})-\frac{K}{K+1})}{2\qty(h^\beta(x_{d,u})-0.5)}+w^{\epsilon,1}_x(x_{d,u})=\rho^{\epsilon,u}_x\qty(x_{d,u}^+),\\
&&\rho^{\epsilon,u}(0)=\alpha+\delta_\alpha+w^{\epsilon,1}(0)-A_1-\delta_0>\alpha,\\
&&\lim_{\epsilon\to0}\rho^{\epsilon,u}(1)=1-\beta+\frac{\delta_0}{3}>1-\beta,\\
&&\varlimsup_{\epsilon\to0}m\qty(\abs{\rho^{\epsilon,u}-\hat{\rho}}>2\delta)\le 
m\qty(\abs{\rho^u-A_3-\hat{\rho}}>\delta)+\varlimsup_{\epsilon\to0}m\qty(\abs{w^{\epsilon,1}-A_1}>\delta)\\
&&=x_d-x_{d,u}<\delta.
\end{eqnarray*}

\paragraph{Lower solution}

$\exists0<\delta_\alpha<\delta$, $\forall x\in[0,1]$, $l^\alpha-l^{\alpha-\delta_\alpha}<\delta$. Then $\alpha-\delta_\alpha>\alpha-\delta>g^\beta(0)$. By Lemma \ref{LemXdEU}, $\exists x_{d,l}\in(0,1)$, $l^{\alpha-\delta_\alpha}(x_{d,l})+h^\beta(x_{d,l})=1$. Because
\begin{eqnarray*}
l^{\alpha-\delta_\alpha}(x_d)+h^\beta(x_d)<l^{\alpha}(x_d)+h^\beta(x_d)=1,
\end{eqnarray*}
by Lemma \ref{LemXdEU}, $x_d<x_{d,l}$. $\exists\delta_\alpha$, $x_{d,l}-x_d<\delta$. By Lemma \ref{LemXdEU}, $l^{\alpha-\delta_\alpha}(x_{d,l}+\delta)+h^\beta(x_{d,l}+\delta)>1$. Let
\begin{eqnarray*}
&&0<\delta_1<\min\qty(\delta,l^{\alpha-\delta_\alpha}(x_{d,l}+\delta)+h^\beta(x_{d,l}+\delta)-1),\\
&&\delta_2=\frac{l^{\alpha-\delta_\alpha}(x_{d,l}+\delta)+h^\beta(x_{d,l}+\delta)-1-\delta_1}{2},\\
&&u=\left\{\begin{array}{ll}
\frac{\delta_2}{2}, & 0\le x\le x_{d,l}+\delta,\\
0, & x_{d,l}+\delta<x\le 1,
\end{array}\right.\\
&&w^{\epsilon,1}(x_{d,l}+\delta)=0.5,\\
&&A_1=\left\{\begin{array}{ll}
l^{\alpha-\delta_\alpha}(x_{d,l}+\delta)-\delta_2, & 0\le x\le x_{d,l}+\delta,\\
h^\beta(x_{d,l}+\delta)-\delta_1-\delta_2, & x_{d,l}+\delta<x\le 1,
\end{array}\right.\\
&&A_3=\left\{\begin{array}{ll}
0, & 0\le x\le x_{d,l}+\delta,\\
\delta_1, & x_{d,l}+\delta<x\le 1,
\end{array}\right.\\
&&\rho^l=\left\{\begin{array}{ll}
l^{\alpha-\delta_\alpha}, & 0\le x\le x_{d,l}+\delta,\\
h^\beta, & x_{d,l}+\delta<x\le 1,
\end{array}\right.\\
&&\rho^{\epsilon,l}=\rho^l+w^{\epsilon,1}-A_1-A_3.
\end{eqnarray*}
By continuity, $\exists 0<\delta_l<x_{d,l}+\delta$,
\begin{eqnarray*}
&&\max_{x\in[x_{d,l}+\delta-\delta_l,x_{d,l}+\delta]}\qty(l^{\alpha-\delta_\alpha}-A_1-A_3+u(w^{\epsilon,1}-0.5))\qty(w^{\epsilon,1}-1+A_1)\\
&&\le\qty(l^{\alpha-\delta_\alpha}(x_{d,l}+\delta-\delta_l)-l^{\alpha-\delta_\alpha}(x_{d,l}+\delta)+\frac{\delta_2}{2})\qty(l^{\alpha-\delta_\alpha}(x_{d,l}+\delta)-\delta_2-0.5)\\
&&<0.
\end{eqnarray*}
By Lemma \ref{LemwPro},
\begin{eqnarray*}
&&\max_{x\in[0,x_{d,l}+\delta-\delta_l]}\abs{w^{\epsilon,1}-A_1}\\
&&\le\exp\bigg(\frac{2}{\epsilon}(2l^{\alpha+\delta_\alpha}(x_{d,l}+\delta)-2\delta_2-1)\delta_l+\log\abs{0.5-l^{\alpha-\delta_\alpha}(x_{d,l}+\delta)+\delta_2}\\
&&-\log\abs{l^{\alpha-\delta_\alpha}(x_{d,l}+\delta)-\delta_2-0.5}+\log(1-2l^{\alpha-\delta_\alpha}(x_{d,l}+\delta)+2\delta_2)\bigg).
\end{eqnarray*}
Because $\min_{x\in[0,x_{d,l}+\delta]}l^{\alpha-\delta_\alpha}_{xx}>0$, we have, $\exists\epsilon_0>0$, $\forall\epsilon<\epsilon_0$,
\begin{eqnarray*}
\min_{x\in[0,x_{d,l}+\delta]}L\rho^{\epsilon,l}>0.
\end{eqnarray*}
By continuity, $\exists 0<\delta_h<1-x_{d,l}-\delta$,
\begin{eqnarray*}
&&\inf_{x\in(x_{d,l}+\delta,x_{d,l}+\delta+\delta_h]}\qty(h^\beta-A_1-A_3)\qty(w^{\epsilon,1}-1+A_1)\\
&&\ge\qty(\min_{x\in[x_{d,l}+\delta,x_{d,l}+\delta+\delta_h]}h^\beta-h^\beta(x_{d,l}+\delta)+\delta_2)\qty(h^\beta(x_{d,l}+\delta)-\delta_1-\delta_2-0.5)\\
&&>0.
\end{eqnarray*}
By Lemma \ref{LemwPro},
\begin{eqnarray*}
&&\max_{x\in[x_{d,l}+\delta+\delta_h,1]}\abs{w^{\epsilon,1}-A_1}\\
&&\le\exp\bigg(\frac{2}{\epsilon}(1-2h^\beta(x_{d,l}+\delta)+2\delta_1+2\delta_2)\delta_h-\log\abs{0.5-l^{\alpha-\delta_\alpha}(x_{d,l}+\delta)+\delta_2}\\
&&+\log\abs{0.5-h^\beta(x_{d,l}+\delta)+\delta_1+\delta_2}+\log(2h^\beta(x_{d,l}+\delta)-2\delta_1-2\delta_2-1)\bigg).
\end{eqnarray*}
Therefore,
\begin{eqnarray*}
\varliminf_{\epsilon\to0}\inf_{x\in(x_{d,l}+\delta,1]}L\rho^{\epsilon,l}\ge-\delta_1\Omega_D(K+1)\frac{0.5-\frac{K}{K+1}}{\max_{x\in[x_{d,l}+\delta,1]}h^\beta-0.5}>0.
\end{eqnarray*}
Moreover,
\begin{eqnarray*}
&&\lim_{\epsilon\to0}\epsilon\rho^{\epsilon,l}_x\qty(\qty(x_{d,l}+\delta)^-)=2\qty(1-\frac{\delta_2}{2})\qty(0.5-l^{\alpha-\delta_\alpha}(x_{d,l}+\delta)+\delta_2)^2\\
&&<2\qty(0.5-l^{\alpha-\delta_\alpha}(x_{d,l}+\delta)+\delta_2)^2=\lim_{\epsilon\to0}\epsilon\rho^{\epsilon,l}_x\qty(\qty(x_{d,l}+\delta)^+),\\
&&\lim_{\epsilon\to0}\rho^{\epsilon,l}(0)=\alpha-\delta_\alpha<\alpha,\\
&&\rho^{\epsilon,l}(1)=1-\beta+w^{\epsilon,1}(1)-1+A_1-\delta_1<1-\beta,\\
&&\varlimsup_{\epsilon\to0}m\qty(\abs{\rho^{\epsilon,l}-\hat{\rho}}>2\delta)\\
&&\le 
m\qty(\abs{\rho^l-A_3-\hat{\rho}}>\delta)+\varlimsup_{\epsilon\to0}m\qty(\abs{w^{\epsilon,1}-A_1}>\delta)=x_{d,l}-x_d+\delta<2\delta.
\end{eqnarray*}

\paragraph{Summary}

$\exists\epsilon_0>0$, $\forall\epsilon<\epsilon_0$, $\rho^{\epsilon,l}<\rho^{\epsilon,u}$. By Theorem \ref{upperlower2W12}, $\forall \epsilon<\epsilon_0$, Eq. \eqref{ellipticequationintroduction} has a $W^{1,2}(0,1)$ solution $\rho^\epsilon$,
\begin{eqnarray*}
\varlimsup_{\epsilon\to 0}m\qty(\abs{\rho^\epsilon-\hat{\rho}}>2\delta)\le 3\delta.
\end{eqnarray*}
Theorem \ref{TheEqOmega} follows Lemma \ref{LemConInMea}.

\subsection{Phase 5: $\beta>0.5$, $g^{0.5}(0)<\alpha<1-h^{0.5}(0)$}

Let $0<\delta<\min\qty(\frac{1-h^{0.5}(0)-\alpha}{2},\alpha-g^{0.5}(0))$.

\paragraph{Upper solution}

$\exists \delta_\alpha>0$, $\forall x\in[0,1]$, $l^{\alpha+\delta_\alpha}-l^\alpha<\delta$. $\exists\delta_\beta>0$, $\forall x\in[0,1]$, $h^{0.5-\delta_\beta}-h^{0.5}<\delta$. Then
\begin{eqnarray*}
\alpha+\delta_\alpha<\alpha+\delta<1-h^{0.5}(0)-\delta<1-h^{0.5-\delta_\beta}(0).
\end{eqnarray*}
By Lemma \ref{LemXdEU}, $\exists x_{d,u}\in(0,1)$, $l^{\alpha+\delta_\alpha}(x_{d,u})+h^{0.5-\delta_\beta}(x_{d,u})=1$. Because
\begin{eqnarray*}
l^{\alpha+\delta_\alpha}(x_d)+h^{0.5-\delta_\beta}(x_d)>l^\alpha(x_d)+h^{0.5}(x_d)=1,
\end{eqnarray*}
by Lemma \ref{LemXdEU}, $x_{d,u}<x_d$. $\exists\qty(\delta_\alpha,\delta_\beta)$, $x_d-x_{d,u}<\delta$. $\exists\delta_0>0$, $\forall x\in[0,1]$, $l^{\alpha+\delta_\alpha}-\delta_0>l^\alpha$. Let $u=0$, $w^{\epsilon,1}(x_{d,u})=0.5$,
\begin{eqnarray*}
&&A_1=\left\{\begin{array}{ll}
l^{\alpha+\delta_\alpha}(x_{d,u})-\frac{\delta_0}{2}, & 0\le x\le x_{d,u},\\
h^{0.5-\delta_\beta}(x_{d,u})+\frac{\delta_0}{2}, & x_{d,u}<x\le 1,
\end{array}\right.\\
&&A_3=\left\{\begin{array}{ll}
\delta_0, & 0\le x\le x_{d,u},\\
0, & x_{d,u}<x\le 1,
\end{array}\right.\\
&&\rho^u=\left\{\begin{array}{ll}
l^{\alpha+\delta_\alpha}, & 0\le x\le x_{d,u},\\
h^{0.5-\delta_\beta}, & x_{d,u}<x\le 1,
\end{array}\right.\\
&&\rho^{\epsilon,u}=\rho^u+w^{\epsilon,1}-A_1-A_3.
\end{eqnarray*}
Because
\begin{eqnarray*}
\min_{x\in[0,x_{d,u}]}\qty(l^{\alpha+\delta_\alpha}-A_1-A_3)\qty(w^{\epsilon,1}-1+A_1)\ge\frac{\delta_0}{2}\qty(0.5-l^{\alpha+\delta_\alpha}(x_{d,u})+\frac{\delta_0}{2})>0,
\end{eqnarray*}
we have
\begin{eqnarray*}
\varlimsup_{\epsilon\to0}\max_{x\in[0,x_{d,u}]}L\rho^{\epsilon,u}\le-\delta_0\Omega_D(K+1)\frac{0.5-\frac{K}{K+1}}{\alpha+\delta_\alpha-0.5}<0.
\end{eqnarray*}
Because
\begin{eqnarray*}
&&\sup_{x\in(x_{d,u},1]}\qty(h^{0.5-\delta_\beta}-A_1-A_3)\qty(w^{\epsilon,1}-1+A_1)\\
&&\le-\frac{\delta_0}{2}\qty(h^{0.5-\delta_\beta}(x_{d,u})+\frac{\delta_0}{2}-0.5)<0,\\
&&\max_{x\in[x_{d,u},1]}h^{0.5-\delta_\beta}_{xx}<0,
\end{eqnarray*}
we have, $\exists\epsilon_0>0$, $\forall\epsilon<\epsilon_0$,
\begin{eqnarray*}
\sup_{x\in(x_{d,u},1]}L\rho^{\epsilon,u}<0.
\end{eqnarray*}
Moreover,
\begin{eqnarray*}
&&\rho^{\epsilon,u}_x\qty(x_{d,u}^-)=\frac{(K+1)\Omega_D\qty(l^{\alpha+\delta_\alpha}(x_{d,u})-\frac{K}{K+1})}{2(l^{\alpha+\delta_\alpha}(x_{d,u})-0.5)}+w^{\epsilon,1}_x(x_{d,u})\\
&&>\frac{(K+1)\Omega_D}{2}+w^{\epsilon,1}_x(x_{d,u})\\
&&>\frac{(K+1)\Omega_D\qty(h^{0.5-\delta_\beta}(x_{d,u})-\frac{K}{K+1})}{2\qty(h^{0.5-\delta_\beta}(x_{d,u})-0.5)}+w^{\epsilon,1}_x(x_{d,u})=\rho^{\epsilon,u}_x\qty(x_{d,u}^+),\\
&&\varlimsup_{\epsilon\to0}m\qty(\abs{\rho^{\epsilon,u}-\hat{\rho}}>2\delta)\le m\qty(\abs{\rho^u-A_3-\hat{\rho}}>\delta)+\varlimsup_{\epsilon\to0}m\qty(\abs{w^{\epsilon,1}-A_1}>\delta)\\
&&=x_d-x_{d,u}<\delta.
\end{eqnarray*}

\paragraph{Lower solution}

$\exists\delta_\alpha>0$, $\forall x\in[0,1]$, $l^\alpha-l^{\alpha-\delta_\alpha}<\delta$. Then $\alpha-\delta_\alpha>\alpha-\delta>g^{0.5}(0)$. By Lemma \ref{LemXdEU}, $\exists x_{d,l}\in(0,1)$, $l^{\alpha-\delta_\alpha}(x_{d,l})+h^{0.5}(x_{d,l})=1$. Because
\begin{eqnarray*}
l^{\alpha-\delta_\alpha}(x_d)+h^{0.5}(x_d)<l^\alpha(x_d)+h^{0.5}(x_d)=1,
\end{eqnarray*}
by Lemma \ref{LemXdEU}, $x_d<x_{d,l}$. $\exists\delta_\alpha$, $x_{d,l}-x_d<\delta$. By Lemma \ref{LemXdEU}, $l^{\alpha-\delta_\alpha}(x_{d,l}+\delta)+h^{0.5}(x_{d,l}+\delta)>1$. Let $\delta_1=\min\qty(\delta,l^{\alpha-\delta_\alpha}(x_{d,l}+\delta)+h^{0.5}(x_{d,l}+\delta)-1)$. $\exists\delta_\beta>0$, $\forall x\in[0,1]$, $h^{0.5-\delta_\beta}-h^{0.5}<\delta_1$. Let
\begin{eqnarray*}
&&\delta_2=\frac{l^{\alpha-\delta_\alpha}(x_{d,l}+\delta)+h^{0.5-\delta_\beta}(x_{d,l}+\delta)-\delta_1-1}{2},\\
&&u=\left\{\begin{array}{ll}
\frac{\delta_2}{2}, & 0\le x\le x_{d,l}+\delta,\\
0, & x_{d,l}+\delta<x\le 1,
\end{array}\right.\\
&&w^{\epsilon,1}(x_{d,l}+\delta)=0.5,\\
&&A_1=\left\{\begin{array}{ll}
l^{\alpha-\delta_\alpha}(x_{d,l}+\delta)-\delta_2, & 0\le x\le x_{d,l}+\delta,\\
h^{0.5-\delta_\beta}(x_{d,l}+\delta)-\delta_1-\delta_2, & x_{d,l}+\delta<x\le 1,
\end{array}\right.\\
&&0<\delta_3<\delta_1-\delta_\beta,\\
&&A_2=0.5-\delta_3,\\
&&w^{\epsilon,2}(1)=1-\beta-\delta_3,\\
&&A_3=\left\{\begin{array}{ll}
0, & 0\le x\le x_{d,l}+\delta,\\
\delta_1, & x_{d,l}+\delta<x\le 1,
\end{array}\right.\\
&&\rho^l=\left\{\begin{array}{ll}
l^{\alpha-\delta_\alpha}, & 0\le x\le x_{d,l}+\delta,\\
h^{0.5-\delta_\beta}, & x_{d,l}+\delta<x\le 1,
\end{array}\right.\\
&&\rho^{\epsilon,l}=\rho^l+w^{\epsilon,1}+w^{\epsilon,2}-A_1-A_2-A_3.
\end{eqnarray*}
By continuity, $\exists0<\delta_l<x_{d,l}+\delta$,
\begin{eqnarray*}
&&\varlimsup_{\epsilon\to0}\max_{x\in[x_{d,l}+\delta-\delta_l,x_{d,l}+\delta]}\qty(l^{\alpha-\delta_\alpha}+w^{\epsilon,2}-A_1-A_2-A_3)\qty(w^{\epsilon,1}-1+A_1)\\
&&\le\qty(l^{\alpha-\delta_\alpha}(x_{d,l}+\delta-\delta_l)-l^{\alpha-\delta_\alpha}(x_{d,l}+\delta)+\delta_2)\qty(l^{\alpha-\delta_\alpha}(x_{d,l}+\delta)-\delta_2-0.5)\\
&&<0.
\end{eqnarray*}
By Lemma \ref{LemwPro},
\begin{eqnarray*}
&&\max_{x\in[0,x_{d,l}+\delta-\delta_l]}\abs{w^{\epsilon,1}-A_1}\\
&&\le\exp\bigg(\frac{2}{\epsilon}(2l^{\alpha-\delta_\alpha}(x_{d,l}+\delta)-2\delta_2-1)\delta_l+\log\abs{0.5-l^{\alpha-\delta_\alpha}(x_{d,l}+\delta)+\delta_2}\\
&&-\log\abs{l^{\alpha-\delta_\alpha}(x_{d,l}+\delta)-\delta_2-0.5}+\log(1-2l^{\alpha-\delta_\alpha}(x_{d,l}+\delta)+2\delta_2)\bigg),\\
&&\max_{x\in[0,x_{d,l}+\delta]}\abs{w^{\epsilon,2}-A_2}\le\exp\bigg(-\frac{4\delta_3}{\epsilon}(1-x_{d,l}-\delta)+\log\abs{0.5-\beta}\bigg).
\end{eqnarray*}
Because $\min_{x\in[0,x_{d,l}+\delta]}l^{\alpha-\delta_\alpha}_{xx}>0$, we have, $\exists\epsilon_0>0$, $\forall\epsilon<\epsilon_0$,
\begin{eqnarray*}
\min_{x\in[0,x_{d,l}+\delta]}L\rho^{\epsilon,l}>0.
\end{eqnarray*}
By continuity, $\exists0<\delta_{h,1},\delta_{h,2}<1-x_{d,l}-\delta$,
\begin{eqnarray*}
&&\varliminf_{\epsilon\to0}\inf_{x\in(x_{d,l}+\delta,x_{d,l}+\delta+\delta_{h,1}]}\qty(h^{0.5-\delta_\beta}+w^{\epsilon,2}-A_1-A_2-A_3)\qty(w^{\epsilon,1}-1+A_1)\\
&&\ge\qty(h^{0.5-\delta_\beta}(x_{d,l}+\delta+\delta_{h,1})-h^{0.5-\delta_\beta}(x_{d,l}+\delta)+\delta_2)\\
&&\qty(0.5-l^{\alpha-\delta_\alpha}(x_{d,l}+\delta)+\delta_2)>0,\\
&&\varliminf_{\epsilon\to0}\min_{x\in[1-\delta_{h,2},1]}\qty(h^{0.5-\delta_\beta}+w^{\epsilon,1}-A_1-A_2-A_3)\qty(w^{\epsilon,2}-1+A_2)\\
&&\ge2\qty(0.5+\delta_1-h^{0.5-\delta_\beta}(1-\delta_{h,2})-\delta_3)\delta_3>0.
\end{eqnarray*}
By Lemma \ref{LemwPro},
\begin{eqnarray*}
&&\max_{x\in[x_{d,l}+\delta+\delta_{h,1},1]}\abs{w^{\epsilon,1}-A_1}\\
&&\le\exp\bigg(\frac{2}{\epsilon}(2l^{\alpha-\delta_\alpha}(x_{d,l}+\delta)-2\delta_2-1)\delta_{h,1}-\log\abs{0.5-l^{\alpha-\delta_\alpha}(x_{d,l}+\delta)+\delta_2}\\
&&+\log\abs{l^{\alpha-\delta_\alpha}(x_{d,l}+\delta)-\delta_2-0.5}+\log(1-2l^{\alpha-\delta_\alpha}(x_{d,l}+\delta)+2\delta_2)\bigg),\\
&&\sup_{x\in(x_{d,l}+\delta,1-\delta_{h,2}]}\abs{w^{\epsilon,2}-A_2}\le\exp\bigg(-\frac{4\delta_3\delta_{h,2}}{\epsilon}+\log\abs{0.5-\beta}\bigg).
\end{eqnarray*}
Therefore,
\begin{eqnarray*}
\varliminf_{\epsilon\to0}\inf_{x\in(x_{d,l}+\delta,1]}L\rho^{\epsilon,l}\ge-\delta_1\Omega_D(K+1)\frac{0.5-\frac{K}{K+1}}{h^{0.5-\delta_\beta}(x_{d,l}+\delta)-0.5}>0.
\end{eqnarray*}
Moreover,
\begin{eqnarray*}
&&\lim_{\epsilon\to0}\epsilon\rho^{\epsilon,l}_x\qty(\qty(x_{d,l}+\delta)^-)=2\qty(1-\frac{\delta_2}{2})(0.5-l^{\alpha-\delta_\alpha}(x_{d,l}+\delta)+\delta_2)^2\\
&&<2(0.5-l^{\alpha-\delta_\alpha}(x_{d,l}+\delta)+\delta_2)^2=\lim_{\epsilon\to0}\epsilon\rho^{\epsilon,l}_x\qty(\qty(x_{d,l}+\delta)^+),\\
&&\lim_{\epsilon\to0}\rho^{\epsilon,l}(0)=\alpha-\delta_\alpha<\alpha,\\
&&\rho^{\epsilon,l}(1)=1-\beta+\delta_\beta+w^{\epsilon,1}(1)-A_1-\delta_1<1-\beta,\\
&&\varlimsup_{\epsilon\to0}m\qty(\abs{\rho^{\epsilon,l}-\hat{\rho}}>2\delta)\le m\qty(\abs{\rho^l-A_3-\hat{\rho}}>\delta)\\
&&+\varlimsup_{\epsilon\to0}m\qty(\abs{w^{\epsilon,1}-A_1+w^{\epsilon,2}-A_2}>\delta)=x_{d,l}-x_d+\delta<2\delta.
\end{eqnarray*}

\paragraph{Summary}

$\exists\epsilon_0>0$, $\forall\epsilon<\epsilon_0$, $\rho^{\epsilon,l}<\rho^{\epsilon,u}$. By Theorem \ref{upperlower2W12}, $\forall \epsilon<\epsilon_0$, Eq. \eqref{ellipticequationintroduction} has a $W^{1,2}(0,1)$ solution $\rho^\epsilon$,
\begin{eqnarray*}
\varlimsup_{\epsilon\to 0}m\qty(\abs{\rho^\epsilon-\hat{\rho}}>2\delta)\le 3\delta.
\end{eqnarray*}
Theorem \ref{TheEqOmega} follows Lemma \ref{LemConInMea}.

\section{Conclusions and Remarks}\label{conclusion}
This paper is devoted to study an initial value parabolic problem with Dirichlet boundary conditions in Eq. \eqref{continuumlimitintroduction}, which originates from the continuum limit of TASEP-LK coupled process. The phase diagram of the steady-state problem in Eq. \eqref{ellipticequationintroduction} is important for understanding both macroscopic and microscopic biological processes, and has been extensively studied by Monte Carlo simulations and numerical computations. We prove many properties of Eqs. \eqref{continuumlimitintroduction},\eqref{ellipticequationintroduction}, including the existence and uniqueness of their solutions and the global attractivity of the steady-state solution. By the method of upper and lower solution, we finally come to the following conclusions.
\begin{enumerate}
  \item Eq. \eqref{ellipticequationintroduction} has a $W^{1,2}(0,1)$ solution with the same phase diagram specified by previous Monte Carlo simulations and numerical computations.
  \item $L^\infty(0,1)$ solution of Eq. \eqref{ellipticequationintroduction} has $C^\infty[0,1]$ regularity.
  \item The solution of Eq. \eqref{ellipticequationintroduction} is unique in $L^\infty(0,1)$.
  \item Eq. \eqref{continuumlimitintroduction} has a unique solution in $C([0,1]\times [0,+\infty))\cap C^{2,1}([0,1]\times (0,+\infty))$ for any continuous initial value, which converges to the solution of Eq. \eqref{ellipticequationintroduction} in $C[0,1]$.
\end{enumerate}

Eqs. \eqref{continuumlimitintroduction}, \eqref{ellipticequationintroduction} studied in this paper are from the simplest case of the TASEP-LK coupled process, in which one species of particles (with the same properties, say speed, attachment and detachment rates, initiation and termination rates {\it etc.}) travel along single one-dimensional track, and during each forward step, particles have only single internal biochemical or biophysical state. In the field of biology and physics, there are actually many general cases. For examples, particles may travel along closed track, have different traveling speeds at different domains of the track, include multiple internal states, switch between different tracks, and/or come from different species. Rich biophysical properties have been obtained by Monte Carlo simulations and numerical computations for many general TASEP-LK coupled processes, but almost no mathematical analysis has been carried out to prove the properties of the corresponding differential equations, or validate the numerical results. In the future, we hope to generalize the methods in this paper to more complex cases, or present more sophisticated methods.

\begin{acknowledgements}
The authors thank Professor Yuan Lou in Ohio State University and Yongqian Zhang in Fudan University for useful discussions.
\end{acknowledgements}

\section*{Data availability}
The data that supports the findings of this study are available within the article.


\bibliographystyle{spmpsci}      % mathematics and physical sciences
\bibliography{reference}   % name your BibTeX data base

%\begin{thebibliography}{}
% etc
%\end{thebibliography}

\end{document}